\documentclass[10pt]{article}

\usepackage{enumerate,xspace}
\usepackage{amsmath,amssymb,wasysym}
\usepackage[all]{xy}
\usepackage{proof}
\usepackage[svgnames]{xcolor}
\usepackage{pict2e} 
\usepackage{tikz}
\usepackage{stmaryrd} 
\usepackage{MnSymbol} 
\usepackage{mathtools}
\usepackage{latexsym}
\usepackage{dsfont}
\usepackage{multicol}
\usepackage{cmll}
\usepackage{multirow}
\usepackage{longtable}

\usepackage{lscape}
\usepackage{array}

\usepackage{hyperref} 
\hypersetup{
    colorlinks,
    citecolor=red,
    filecolor=red,
    linkcolor=blue,
    urlcolor=red
}

\newtheorem{observation}{Remark}[section]
\newtheorem{lemma}[observation]{Lemma}  
\newtheorem{theorem}[observation]{Theorem}
\newtheorem{definition}[observation]{Definition}
\newtheorem{example}[observation]{Example}

\newtheorem{proposition}[observation]{Proposition} 
\newtheorem{corollary}[observation]{Corollary}

\makeatletter

\newdimen\w@dth

\def\setw@dth#1#2{\setbox\z@\hbox{\scriptsize $#1$}\w@dth=\wd\z@
\setbox\@ne\hbox{\scriptsize $#2$}\ifnum\w@dth<\wd\@ne \w@dth=\wd\@ne \fi
\advance\w@dth by 1.2em}

\def\t@^#1_#2{\allowbreak\def\n@one{#1}\def\n@two{#2}\mathrel
{\setw@dth{#1}{#2}
\mathop{\hbox to \w@dth{\rightarrowfill}}\limits
\ifx\n@one\empty\else ^{\box\z@}\fi
\ifx\n@two\empty\else _{\box\@ne}\fi}}
\def\t@@^#1{\@ifnextchar_ {\t@^{#1}}{\t@^{#1}_{}}}

\def\t@left^#1_#2{\def\n@one{#1}\def\n@two{#2}\mathrel{\setw@dth{#1}{#2}
\mathop{\hbox to \w@dth{\leftarrowfill}}\limits
\ifx\n@one\empty\else ^{\box\z@}\fi
\ifx\n@two\empty\else _{\box\@ne}\fi}}
\def\t@@left^#1{\@ifnextchar_ {\t@left^{#1}}{\t@left^{#1}_{}}}

\def\two@^#1_#2{\def\n@one{#1}\def\n@two{#2}\mathrel{\setw@dth{#1}{#2}
\mathop{\vcenter{\hbox to \w@dth{\rightarrowfill}\kern-1.7ex
                 \hbox to \w@dth{\rightarrowfill}}%
       }\limits
\ifx\n@one\empty\else ^{\box\z@}\fi
\ifx\n@two\empty\else _{\box\@ne}\fi}}
\def\tw@@^#1{\@ifnextchar_ {\two@^{#1}}{\two@^{#1}_{}}}

\def\tofr@^#1_#2{\def\n@one{#1}\def\n@two{#2}\mathrel{\setw@dth{#1}{#2}
\mathop{\vcenter{\hbox to \w@dth{\rightarrowfill}\kern-1.7ex
                 \hbox to \w@dth{\leftarrowfill}}%
       }\limits
\ifx\n@one\empty\else ^{\box\z@}\fi
\ifx\n@two\empty\else _{\box\@ne}\fi}}
\def\t@fr@^#1{\@ifnextchar_ {\tofr@^{#1}}{\tofr@^{#1}_{}}}

\newdimen\W@dth
\def\setW@dth#1#2{\setbox\z@\hbox{$#1$}\W@dth=\wd\z@
\setbox\@ne\hbox{$#2$}\ifnum\W@dth<\wd\@ne \W@dth=\wd\@ne \fi
\advance\W@dth by 1.2em}

\def\T@^#1_#2{\allowbreak\def\N@one{#1}\def\N@two{#2}\mathrel
{\setW@dth{#1}{#2}
\mathop{\hbox to \W@dth{\rightarrowfill}}\limits
\ifx\N@one\empty\else ^{\box\z@}\fi
\ifx\N@two\empty\else _{\box\@ne}\fi}}
\def\T@@^#1{\@ifnextchar_ {\T@^{#1}}{\T@^{#1}_{}}}

\def\T@left^#1_#2{\def\N@one{#1}\def\N@two{#2}\mathrel{\setW@dth{#1}{#2}
\mathop{\hbox to \W@dth{\leftarrowfill}}\limits
\ifx\N@one\empty\else ^{\box\z@}\fi
\ifx\N@two\empty\else _{\box\@ne}\fi}}
\def\T@@left^#1{\@ifnextchar_ {\T@left^{#1}}{\T@left^{#1}_{}}}

\def\Tofr@^#1_#2{\def\N@one{#1}\def\N@two{#2}\mathrel{\setW@dth{#1}{#2}
\mathop{\vcenter{\hbox to \W@dth{\rightarrowfill}\kern-1.7ex
                 \hbox to \W@dth{\leftarrowfill}}%
       }\limits
\ifx\N@one\empty\else ^{\box\z@}\fi
\ifx\N@two\empty\else _{\box\@ne}\fi}}
\def\T@fr@^#1{\@ifnextchar_ {\Tofr@^{#1}}{\Tofr@^{#1}_{}}}

\def\Two@^#1_#2{\def\N@one{#1}\def\N@two{#2}\mathrel{\setW@dth{#1}{#2}
\mathop{\vcenter{\hbox to \W@dth{\rightarrowfill}\kern-1.7ex
                 \hbox to \W@dth{\rightarrowfill}}%
       }\limits
\ifx\N@one\empty\else ^{\box\z@}\fi
\ifx\N@two\empty\else _{\box\@ne}\fi}}
\def\Tw@@^#1{\@ifnextchar_ {\Two@^{#1}}{\Two@^{#1}_{}}}

\def\to{\@ifnextchar^ {\t@@}{\t@@^{}}}
\def\from{\@ifnextchar^ {\t@@left}{\t@@left^{}}}
\def\tofro{\@ifnextchar^ {\t@fr@}{\t@fr@^{}}}
\def\To{\@ifnextchar^ {\T@@}{\T@@^{}}}
\def\From{\@ifnextchar^ {\T@@left}{\T@@left^{}}}
\def\Two{\@ifnextchar^ {\Tw@@}{\Tw@@^{}}}
\def\Tofro{\@ifnextchar^ {\T@fr@}{\T@fr@^{}}}

\makeatother

\title{Classical Distributive Restriction Categories}
\author{Robin Cockett and Jean-Simon Pacaud Lemay}

\begin{document}
\allowdisplaybreaks

\maketitle

\begin{abstract} In the category of sets and partial functions, $\mathsf{PAR}$, while the disjoint union $\sqcup$ is the usual categorical coproduct, the Cartesian product $\times$ becomes a restriction categorical analogue of the categorical product: a restriction product. Nevertheless, $\mathsf{PAR}$ does have a usual categorical product as well in the form 
$A \& B := A \sqcup B \sqcup (A \times B)$. Surprisingly, asking that a {\em distributive\/} restriction category (a restriction category with restriction products $\times$ and coproducts $\oplus$) has $A \& B$ a categorical product is enough to imply that the category is a {\em classical\/} restriction category.  This is a restriction category which has joins and relative complements and,  thus, supports classical Boolean reasoning.  The first and main observation of the paper is that a distributive restriction category is classical if and only if $A \& B := A \oplus B \oplus (A \times B)$ is a categorical product in which case we call $\&$ the ``classical'' product.

In fact, a distributive restriction category has a categorical product if and only if it is a {\em classified\/} restriction category.  This is in the sense that every map $A \to B$ factors uniquely through a total map $A \to B \oplus \mathsf{1}$, where $\mathsf{1}$ is the restriction terminal object. This implies the second significant observation of the paper, namely, that a distributive restriction category has a classical product if and only if it is the Kleisli category of the exception monad $\_ \oplus \mathsf{1}$ for an ordinary distributive category.  

Thus having a classical product has a significant structural effect on a distributive restriction category.  In particular, the classical product not only provides an alternative axiomatization for being classical but also for being the Kleisli category of the exception monad on an ordinary distributive category.   \end{abstract}

\noindent \small \textbf{Acknowledgements.} The authors would like to thank Masahito Hasegawa and RIMS at Kyoto University for helping fund research visits so that the authors could work together on this project. For this research, the first named author was partially funded by NSERC, while the second named author was financially supported by a JSPS Postdoctoral Fellowship, Award \#: P21746. \\ 

We would like to dedicate this paper to Pieter Hofstra (1975-2022). 

\paragraph{From Robin:} I met Pieter in 2004, while I was visiting the University of Ottawa. I was assigned a desk in the postdoc's office of the math department and Pieter was a postdoc there at the time.  On the pretext that we would develop a restriction categorical version of recursion theory, I persuaded Pieter to come to Calgary. He actually did come and he was a postdoc in Computer Science at the University of Calgary 2005-2007.  Furthermore, we actually did introduce a formulation of recursion theory in a restriction category which we called a Turing category \cite{turing-categories}.  I am immensely proud of this foundational work. Pieter brought an uncompromising high standard of exposition to the table ... and a deep understanding of recursion theory.  We ended up writing a number of joint papers some of which are still not finished. 
Much of our work together involved restriction categories, which is, appropriately, the subject of this paper. 

Pieter, besides being a highly valued colleague, was a frequent visitor to our home and he and his family became close friends to my wife, Polly, and I.  He is greatly missed.

\paragraph{From JS:} Pieter played an important role in my life. Pieter was my first academic mentor: he taught me category theory and I did my first research project with him while I was an undergrad student at the University of Ottawa. In fact, the last lecture of my undergrad was with Pieter, where he taught us how to kill Hydras. Afterwards, Pieter became a great colleague and friend, with whom I enjoyed spending time with whenever we crossed paths at conferences and workshops. What I enjoyed most about Pieter is while we did have many interesting discussions about math, we also had great discussions about numerous other subjects during hikes or over food and drinks. Pieter will be greatly missed by myself and many others. May he rest in peace.

\tableofcontents


\section{Introduction}

In the category of sets and partial functions, $\mathsf{PAR}$, it is well known that while the disjoint union of sets $A \sqcup B$ is the coproduct, the Cartesian product of sets $A \times B$ on the other hand is no longer a categorical product. It is a curiosity that there is, nonetheless, a real categorical product in $\mathsf{PAR}$ which takes the form $A  \sqcup B \sqcup (A \times B)$. This raises the general question of what properties are required of a partial map category to come equipped with a categorical product of this form.

To answer this question it is expedient to utilize the formulation of partial map categories as restriction categories: these were introduced by Cockett and Lack in \cite{cockett2002restriction}. Briefly, a restriction category (Def. \ref{def:restcat}) is a category equipped with a restriction operator which associates to every map $f$ its restriction $\overline{f}$, capturing the domain of definition of $f$. The main example of a restriction category is $\mathsf{PAR}$ (Ex. \ref{ex:PAR}), but there are also many other interesting examples including the opposite category of commutative algebras and \emph{non-unital} algebra morphisms (Ex. \ref{ex:CALG}). 

As mentioned, while in $\mathsf{PAR}$ the Cartesian product $\times$ is not a categorical product, it is instead a restriction product (Def \ref{def:restprod}): this is the restriction categorical analogue of a product which is, in fact, a lax limit.  This means that it is, in general, not a categorical product as it need not be a strict limit. On the other hand, the notion of a restriction coproduct $\oplus$ is a coproduct in the usual categorical sense (Def \ref{def:restcoprod}). A \emph{distributive} restriction category (Def \ref{def:drc}) is a restriction category with both restriction products and coproducts, such that the restriction product distributes over the coproduct in the sense that the natural map $(A \times B) \oplus (A \times C) \to A \times (B \oplus C)$ is an isomorphism. 

As was discussed in \cite[Sec 4.3]{cockett2007restriction}, there is no reason why a restriction category should not possess a categorical product. In fact, it may be tempting to think that in a distributive restriction category, $A \oplus B \oplus (A \times B)$ will always be a categorical product. However, this is not so. Indeed, any category is a restriction category with the trivial restriction operator where for any map $f$, $\overline{f}$ is the simply identity on the domain of $f$. Thus when a distributive restriction category has a trivial restriction, in this sense, the restriction product $\times$ is already a categorical product and it is not of the desired form. This means that the question of when a distributive restriction category has ${A \oplus B \oplus (A \times B)}$ a categorical product, is nontrivial.  Surprisingly, as we shall see, demanding that ${A \oplus B \oplus (A \times B)}$ be a categorical product forces a distributive restriction category to be both \emph{classical} and \emph{classified}. 

Classical restriction categories were explored by Cockett and Manes in \cite{cockett2009boolean}. In any restriction category there is a canonical partial order $\leq$ on parallel maps (Def. \ref{def:leq}). A restriction category is classical if it has joins (Def. \ref{def:join}) and relative complements (Def. \ref{def:classrest}) with respect to $\leq$. These are both desirable notions to have in a restriction category as they allow for Boolean classical reasoning. The main result of this paper is that a distributive restriction category is classical if and only if $A \oplus B \oplus (A \times B)$ is a categorical product (Thm. \ref{thm:cdrc}). Consequently, when $A \oplus B \oplus (A \times B)$ is a categorical product we refer to it as a \emph{classical}  product, which we denote as $A \& B = A \oplus B \oplus (A \times B)$ to distinguish it from the restriction product. This allows us to restate our main result: a distributive restriction category is classical if and only if it has classical products. To prove that a classical distributive restriction category has classical products (Prop. \ref{prop:classtoclassprod}), we use joins and complements of restriction idempotents (Def. \ref{def:ecomp}) to construct the necessary unique map $C \to A \& B$ for the universal property of the product. Conversely, to show that having classical products implies being classical (Prop. \ref{prop:cptoc}), we use the universal property of $\&$ to build the joins and relative complements, and make use of the notion of decisions (Def. \ref{def:decision}) in the proof. Thus classical products give a novel and somewhat unexpected way of axiomatizing classical distributive restriction categories.

Important examples of distributive restriction categories are Kleisli categories for the exception monad\footnote{The exception monad has been popularized in Haskell as the ``maybe'' monad.}. Indeed, recall that a distributive category \cite{carboni1993distributive,cockett1993distributive} is a category $\mathbb{D}$ that has products $\times$ and coproducts $\oplus$ (in the usual categorical sense) where $\times$ distributes over $\oplus$. The exception monad of a distributive category $\mathbb{D}$ is the monad defined on $\_ \oplus \mathsf{1}$, where $\mathsf{1}$ is the terminal object. For any distributive category $\mathbb{D}$, the Kleisli category of its exception monad, $\mathbb{D}_{\_ \oplus \mathsf{1}}$, is a distributive restriction category (Prop \ref{prop:d+1class}), where the (co)product $\times$ (resp. $\oplus$) of $\mathbb{D}$ becomes the restriction (co)product in $\mathbb{D}_{\_ \oplus \mathsf{1}}$. Furthermore, it is already known that $A \oplus B \oplus (A \times B)$ is a categorical product in $\mathbb{D}_{\_ \oplus \mathsf{1}}$ \cite[Prop. 3.4]{cockett1997weakly}. Thus, $\mathbb{D}_{\_ \oplus \mathsf{1}}$ has classical products and is therefore also classical. Less obviously, the converse is also true. 

To prove this, we use the concept of classified restriction categories, which were introduced by Cockett and Lack in \cite{cockett2003restriction}. Briefly, a restriction category is classified if it has an abstract partial map classifier and every map factors through this classifier via a unique total map (Def. \ref{def:restclass}). It turns out every classified restriction category is equivalent to the Kleisli category of the canonical monad induced on its subcategory of total maps (Prop. \ref{prop:classkleisli}). From this, we obtain that a distributive restriction category is the Kleisli category of the exception monad of a distributive category if and only if it is {\em classically\/} classified (Def. \ref{def:classclass}) in the sense that every map $f: A \to B$ factors uniquely as a total map $\mathcal{T}(f): A \to B \oplus \mathsf{1}$ (Cor. \ref{cor:classclass}). Thus a distributive restriction category is classical if and only if it is classically classified if and only if is equivalent to the Kleisli category of an exception monad on an ordinary distributive category (Thm. \ref{thm2}). Therefore, a distributive restriction category has classical products precisely when it is the Kleisli category of the exception monad on an ordinary distributive category.

In this manner, for a distributive restriction category to have classical products embodies some surprisingly strong structural properties.

\paragraph{Outline:} Section \ref{sec:drc} reviews the basics of distributive restriction categories. Section \ref{sec:classprod} introduces the notion of classical products for distributive restriction categories. Section \ref{sec:classical} provides a review of classical restriction categories, as well as some new observations about the complements of restriction idempotents. In Section \ref{sec:split} we show how restriction (co)products and categorical products (no matter their form) are related via splitting restriction idempotent, and how in a classical restriction category, the categorical product must always be of the form $A \oplus B \oplus (A \times B)$. In Section \ref{sec:CDRC} we prove the main result of this paper: distributive restriction categories are classical if and only if they have classical products. Finally, in Section \ref{sec:classified}, we extend the main result and show that a distributive restriction category is classical if and only if it is equivalent to the Kleisli category of the exception monad on a distributive category. 

\paragraph{Conventions} In an arbitrary category $\mathbb{X}$, we denote objects by capital letters $A$, $B$, etc. and maps by lower case letters $f$, $g$, $h$, etc. Identity maps are written as $1_A: A \to A$, while composition is written in diagrammatic order, that is, the composite of maps $f: A \to B$ and $g: B \to C$ is denoted $fg: A \to C$, which first does $f$ then $g$. 

\section{Distributive Restriction Categories}\label{sec:drc}

In this background section, we review distributive restriction categories -- mostly to introduce notation and terminology. As such, we will quickly review restriction categories, restriction products, and restriction coproducts, as well as some basic identities and our two main running examples. For a more in-depth introduction to restriction categories, we invite the reader to see \cite{cockett2002restriction,cockett2003restriction,cockett2007restriction, cockett2009boolean}. 

\begin{definition}\label{def:restcat} A \textbf{restriction category} \cite[Sec 2.1.1]{cockett2002restriction} is a category $\mathbb{X}$ equipped with a \textbf{restriction operator} $\overline{(\phantom{f})}$, which associates every map $f: A \to B$ to a map ${\overline{f}: A \to A}$, called the \textbf{restriction of $f$}, and such that the following four axioms hold: 
\begin{enumerate}[{\bf [R.1]}]
\item $ \overline{f}f = f$
\item $\overline{f} \overline{g} = \overline{g} \overline{f}$
\item $\overline{\overline{g}f} = \overline{g} \overline{f}$
\item $f\overline{g} =  \overline{fg} f $
\end{enumerate}
\noindent Furthermore, in a restriction category $\mathbb{X}$: 
\begin{enumerate}[{\em (i)}]
\item A \textbf{total} map \cite[Sec 2.1.2]{cockett2002restriction} is a map $f: A \to B$ such that $\overline{f} = 1_A$. We denote $\mathcal{T}\left[\mathbb{X} \right]$ to be the subcategory of total maps of $\mathbb{X}$.
\item A \textbf{restriction idempotent} \cite[Sec 2.3.3]{cockett2002restriction} is a map $e: A \to A$ such that $\overline{e} = e$. 
\end{enumerate}
\end{definition}

The canonical example of a restriction category is the category of sets and partial functions, which we review in Ex \ref{ex:PAR} below. For a list of many other examples of restriction categories, see \cite[Sec 2.1.3]{cockett2002restriction}. Here are some basic identities that will be useful for the proofs in this paper: 

\begin{lemma}\cite[Lemma 2.1 \& 2.2]{cockett2002restriction} \label{lem:rest} In a restriction category $\mathbb{X}$: 
\begin{enumerate}[{\em (i)}]
\item \label{lem:rest.i} $\overline{fg}= \overline{f\overline{g}}$
\item \label{lem:rest.total.2} If $g$ is total, then $\overline{fg} = \overline{f}$ 
\item \label{lem:rest.total} If $f$ is monic, then $f$ is total; 
\item \label{lem:rest.idem.R1}  If $e$ is a restriction idempotent, then it is an idempotent, that is, $e e = e$;
\item \label{lem:rest.idem.R2} If $e$ and $e^\prime$ are restriction idempotents of the same type, then their composite is a restriction idempotent and $e e^\prime = e^\prime e$; 
\item \label{lem:rest.idem.R3} If $e$ is a restriction idempotent, then $\overline{ef} = e \overline{f} = \overline{f} e$; 
\item  \label{lem:rest.idem.R4} If $e$ is a restriction idempotent, then $fe = \overline{fe} f$;
\item \label{lem:rest.idem} For any map $f$, $\overline{f}$ is a restriction idempotent, so $\overline{\overline{f}} = \overline{f}$ and $\overline{f}~\overline{f} = \overline{f}$ 
 \end{enumerate}   
\end{lemma}

We now turn our attention to restriction products. As mentioned in the introduction, restriction products are not products in the usual sense, since they satisfy a lax universal property. A 2-categorical explanation of why restriction products are the appropriate analogue of products for restriction categories can be found in \cite[Sec 4.1]{cockett2007restriction} (briefly it is what arises when one considers Cartesian objects in the appropriate 2-category of restriction categories).

\begin{definition}\label{def:restprod} A \textbf{Cartesian restriction category} is a restriction category $\mathbb{X}$ with: 
\begin{enumerate}[{\em (i)}]
\item A \textbf{restriction terminal object} \cite[Sec 4.1]{cockett2007restriction}, that is, an object $\mathsf{1}$ such that for every object $A$ there exists a unique total map $t_A: A \to \mathsf{1}$ such that for every map $f: A \to B$, the following equality holds: 
\begin{align}\label{diag:terminal}
    f t_B = \overline{f} t_A
\end{align}
\item Binary \textbf{restriction products} \cite[Sec 4.1]{cockett2007restriction}, that is, every pair of objects $A$ and $B$, there is an object $A \times B$ with total maps $\pi_0: A \times B \to A$ and $\pi_1: A \times B \to B$ such that for every pair of maps $f: C \to A$ and ${g: C \to B}$, there exists a unique map $\langle f, g \rangle: C \to A \times B$ such that the following equalities hold: 
\begin{align}\label{diag:pair1}
    \langle f, g \rangle \pi_0 = \overline{g}f && \langle f, g \rangle \pi_1 = \overline{f}g
\end{align}
\end{enumerate}
\end{definition}

On the other hand, restriction coproducts are actual coproducts in the usual sense. The justification for this is that actual coproducts are required if one wishes that a restriction category admits a ``calculus of matrices'' \cite[Sec 2.3]{cockett2007restriction}, as well as providing the appropriate restriction analogue of an extensive category \cite[Def 3]{cockett2007restriction} (which we discuss in Def \ref{def:decision} below). A 2-categorical explanation for why restriction coproducts are simply coproducts can be found in \cite[Sec 2.1]{cockett2007restriction} (briefly it is what arises when one considers coCartesian objects in the appropriate 2-category of restriction categories).

In this paper, while we will only need to work with binary restriction products, for simplicity, it will be easier to work with finite (restriction) coproducts. So for a category $\mathbb{X}$ with finite coproducts, we denote the coproduct as $\oplus$, with injection maps ${\iota_j: A_j \to A_0 \oplus \hdots \oplus A_n}$, where the copairing operation is denoted by $[-,\hdots,-]$, and we denote the initial object as $\mathsf{0}$ with unique map $z_A: \mathsf{0} \to A$. 

\begin{definition}\label{def:restcoprod} A \textbf{coCartesian restriction category} is a restriction category $\mathbb{X}$ with finite \textbf{restriction coproducts} \cite[Sec 2.1]{cockett2007restriction}, that is, $\mathbb{X}$ has finite coproducts where all the injection maps ${\iota_j: A_j \to A_0 \oplus \hdots \oplus A_n}$ are total. 
\end{definition}

Lastly, we may ask that restriction products distribute over restriction coproducts:  

\begin{definition} \label{def:drc} A \textbf{distributive restriction category} \cite[Sec 5.3]{cockett2007restriction} is a restriction category $\mathbb{X}$ which is both a Cartesian restriction category and a coCartesian restriction category such that the canonical maps:
\begin{align}
    \left \langle 1_A \times \iota_0,  1_A \times \iota_1 \right \rangle: (A \times B) \oplus (A \times C) \to A \times (B \oplus C) && \langle z_A, 1_\mathsf{0} \rangle: \mathsf{0} \to A \times \mathsf{0}
\end{align}
are isomorphisms. 
\end{definition}

Here are our main examples of distributive restriction categories, which we will use as running examples throughout this paper: 

\begin{example} \normalfont\label{ex:PAR} Let $\mathsf{PAR}$ be the category whose objects are sets and whose maps are partial functions between sets. Then $\mathsf{PAR}$ is a distributive restriction category where: 
\begin{enumerate}[{\em (i)}]
\item The restriction of a partial function $f: X \to Y$ is the partial function $\overline{f}: X \to X$ defined as follows: 
\[ \overline{f}(x) = \begin{cases} x & \text{if } f(x) \downarrow \\
\uparrow & \text{if } f(x) \uparrow \end{cases} \]
where $\downarrow$ means defined and $\uparrow$ means undefined. 
\item The restriction terminal object is a chosen singleton $\mathsf{1} = \lbrace \ast \rbrace$, and $t_X: X \to \lbrace \ast \rbrace$ maps everything to the single element, $t_X(x) = \ast$.
\item The restriction product is given by the Cartesian product $A \times B$, where the projections $\pi_0: X \times Y \to X$ and $\pi_1: X \times Y \to Y$ are defined as $\pi_0(x,y) =x$ and $\pi_1(x,y)=y$, and the pairing of partial functions is defined as: 
\[ \langle f,g \rangle (x) = \begin{cases} (f(x), g(x)) & \text{if } f(x) \downarrow \text{ and } g(x) \downarrow \\
\uparrow & \text{otherwise}  \end{cases} \]
\item The initial object is the empty set $\mathsf{0} = \emptyset$ and the coproduct is given by disjoint union $X \oplus Y = X \sqcup Y$.
\end{enumerate}
The total maps are precisely functions that are everywhere defined, so $\mathcal{T}\left[ \mathsf{PAR} \right] = \mathsf{SET}$, the category of sets and functions. Restriction idempotents of type $X \to X$ correspond precisely to the subsets of $X$. Explicitly, given a subset $U \subseteq X$, its associated restriction idempotent is the partial function $e_U: X \to X$ defined as follows: 
\[ e_U = \begin{cases} x & \text{if } x \in U\\
\uparrow & \text{if } x \notin U \end{cases} \]
In particular, for a partial function $f: X \to Y$, we have that its restriction $\overline{f}$ is the restriction idempotent associated to the domain of $f$, $\mathsf{dom}(f) = \lbrace x \in X \vert~ f(x) \downarrow \rbrace \subseteq X$, that is, $\overline{f} = e_{\mathsf{dom}(f)}$. 
\end{example}

\begin{example} \normalfont\label{ex:CALG}  Let $k$ be a commutative ring and let $k\text{-}\mathsf{CALG}_\bullet$ be the category whose objects are commutative $k$-algebras and whose maps are \emph{non-unital} $k$-algebra morphisms, that is, $k$-linear morphisms $f: A \to B$ that preserve the multiplication, $f(ab) = f(a) f(b)$, but not necessarily the multiplicative unit, so $f(1)$ may not equal $1$. Then $k\text{-}\mathsf{CALG}^{op}_\bullet$ is a distributive restriction category, so $k\text{-}\mathsf{CALG}_\bullet$ is a codistributive corestriction category where:
\begin{enumerate}[{\em (i)}]
\item The corestriction of a non-unital $k$-algebra morphism $f: A \to B$ is the non-unital $k$-algebra morphism ${\overline{f}: B \to B}$ defined as:
\[\overline{f}(b)= f(1)b\]  
\item The corestriction initial object is $k$, and $t_A: k \to A$ is defined by the $k$-algebra structure of $A$. 
\item The corestriction coproduct is given by the tensor product $A \otimes B$, where the injections $\iota_0: A \to A \otimes B$ and ${\iota_1: B \to A \otimes B}$ are defined as $\iota_0(a) = a \otimes 1$ and $\iota_1(b) = 1 \otimes b$, and where the copairing is given by:
\[[f,g](a \otimes b) = f(a)g(b)\]  
\item The terminal object is the zero algebra $\mathsf{0}$ and the product is given by the product of $k$-algebras $A \times B$. 
\end{enumerate}
The total maps correspond precisely to the maps that do preserve the multiplicative unit, $f(1)=1$. In other words, the total maps are the actual $k$-algebra morphisms. So $\mathcal{T}[k\text{-}\mathsf{CALG}^{op}_\bullet] = k\text{-}\mathsf{CALG}^{op}$, where $k\text{-}\mathsf{CALG}$ is the category of commutative $k$-algebras and $k$-algebra morphisms. Restriction idempotents in $k\text{-}\mathsf{CALG}^{op}$, so corestriction idempotents in $k\text{-}\mathsf{CALG}_\bullet$, of type $A \to A$ correspond precisely to idempotent elements of $A$, that is, elements $u \in A$ such that $u^2=1$. Explicitly, for an idempotent element $u \in A$, its associated corestriction idempotent is the non-unital $k$-algebra morphism $e_u: A \to A$ defined as:
\[ e_u(a)=ua \]
So for a non-unital $k$-algebra morphism $f: A \to B$, the corresponding idempotent element of its corestriction $\overline{f}$ is $f(1) \in B$, and so $\overline{f}= e_{f(1)}$. 
\end{example}

\section{Classical Products}\label{sec:classprod}

In this section, we introduce classical products for distributive restriction categories, which is the main concept of interest in this paper. Classical products are actual products in the usual sense. The nomenclature is justified since in Sec \ref{sec:CDRC} we will show that a distributive restriction category has classical products if and only if it is a classical restriction category (which we review in Sec \ref{sec:classical}).  

To distinguish between products and restriction products, we will use linear logic inspired notation for the former. So for a category $\mathbb{X}$ with finite products, we denote the binary product by $\&$, the projection maps as $p_0: A \& B \to A$ and $p_1: A \& B \to B$, and the pairing operation by $\llangle -,- \rrangle$. Concretely, if only to avoid confusion with the restriction product, for every pair of maps $f: C \to A$ and $g: C \to B$, $\llangle f,g \rrangle: C \to A\&B$ is the unique map such that the following equalities hold:
\begin{align}\label{diag:pair2}
    \llangle f,g \rrangle p_0 =f && \llangle f,g \rrangle p_1 = g 
\end{align}
We stress that, unlike the projections of restriction products, the projections $p_i$ are not assumed to be total. We denote the terminal object by $\mathsf{0}$ and unique maps to it by $!_A: A \to \mathsf{0}$. Keen-eyed readers will note that we are using the same notation for both the terminal object and the restriction initial object. This is justified in Lemma \ref{lem:restcoprodzero}.(\ref{lem:restcoprodzero.i}) below. 

Here are some basic identities regarding compatibility between the restriction and the product structure: 

\begin{lemma} \label{lem:rest-product} In a restriction category $\mathbb{X}$ with finite products:
\begin{enumerate}[{\em (i)}]
\item $f \overline{!_B} = \overline{!_A} f$; 
\item \label{lem:rest-product.ii} $\left\llangle f,g \right\rrangle  \overline{p_0} = \left\llangle f, \overline{f}g \right\rrangle$ and $\left\llangle f,g \right\rrangle  \overline{p_1} = \left\llangle \overline{g}f, g \right\rrangle$;
\item \label{lem:rest-product.iii} $\left\llangle f,g \right\rrangle \overline{p_0}~\overline{p_1} = \left\llangle \overline{g}f, \overline{f}g \right\rrangle$
\end{enumerate}
\end{lemma}
\begin{proof} Recall that the pairing $\llangle -, - \rrangle$ is compatible with composition in the following sense: 
\begin{align}\label{eq:pair-comp}
   h \llangle f,g \rrangle (p \& q) = \left \llangle hfk, hgk \right \rrangle 
\end{align}
\begin{enumerate}[{\em (i)}]
\item Here we use \textbf{[R.4]}: 
\begin{align*}
    f \overline{!_B} &=~ \overline{f !_B} f \tag*{\textbf{[R.4]}} \\
    &=~  \overline{!_A} f \tag*{Uniqueness of $!$}
\end{align*}
\item Here  we use \textbf{[R.4]} and \textbf{[R.1]}: 
\begin{align*}
  \left\llangle f,g \right\rrangle  \overline{p_0} &=~ \overline{\left\llangle f,g \right\rrangle p_0} \left\llangle f,g \right\rrangle \tag*{\textbf{[R.4]}} \\
  &=~ \overline{f} \left\llangle f,g \right\rrangle \tag*{(\ref{diag:pair2})} \\ 
&=~ \left\llangle \overline{f}f, \overline{f}g \right\rrangle \tag*{(\ref{eq:pair-comp})} \\ 
&=~ \left\llangle f, \overline{f}g \right\rrangle \tag*{\textbf{[R.4]}}
\end{align*}
So $ \left\llangle f,g \right\rrangle  \overline{p_0} = \left\llangle f, \overline{f}g \right\rrangle$, and similarly we can compute that $\left\llangle f,g \right\rrangle  \overline{p_1} = \left\llangle \overline{g}f,g \right\rrangle$.
\item Here we apply the above two identities, and also \textbf{[R.2]} and \textbf{[R.3]}:
\begin{align*}
    \left\llangle f,g \right\rrangle \overline{p_0}~\overline{p_1} &=~ \left\llangle f, \overline{f}g \right\rrangle \overline{p_1} \tag*{Lemma \ref{lem:rest-product}.(\ref{lem:rest-product.ii})} \\
    &=~  \left\llangle \overline{\overline{f}g} f, \overline{f}g \right\rrangle \tag*{Lemma \ref{lem:rest-product}.(\ref{lem:rest-product.ii})} \\
    &=~  \left\llangle \overline{f} \overline{g} f, \overline{f}g \right\rrangle \tag*{\textbf{[R.3]}} \\
    &=~ \left\llangle \overline{g} \overline{f} f, \overline{f}g \right\rrangle \tag*{\textbf{[R.2]}} \\
    &=~  \left\llangle \overline{g} f, \overline{f}g \right\rrangle \tag*{\textbf{[R.1]}} 
\end{align*}
\end{enumerate}
So the desired identities hold. 
\end{proof}

We now turn our attention to properly defining classical products. So in a distributive restriction category, we wish to know when $A \oplus B \oplus (A \times B)$ is a product of $A$ and $B$. In order to define the projections, we will need restriction zero maps. Intuitively, a restriction zero map is a map which is nowhere defined. For a category $\mathbb{X}$ with zero maps, we denote the zero maps by $0: A \to B$, and recall that zero maps satisfy the annihilation property that $f0=0=0f$ for all maps $f$. 

\begin{definition} A restriction category $\mathbb{X}$ is said to have \textbf{restriction zeroes} \cite[Sec 2.2]{cockett2007restriction} if $\mathbb{X}$ has zero maps such that $\overline{0} =0$.
\end{definition}

\begin{example} \normalfont\label{ex:PAR0} $\mathsf{PAR}$ has restriction zeroes, where the restriction zero maps are the partial functions $0: X \to Y$ which is nowhere defined, $0(x) \uparrow$ for all $x \in X$. 
\end{example}

\begin{example} \normalfont\label{ex:CALG0} $k\text{-}\mathsf{CALG}^{op}_\bullet$ has restriction zeroes, so $k\text{-}\mathsf{CALG}_\bullet$ has corestriction zeroes, where $0: A \to B$ is the zero morphism, $0(a) = 0$. 
\end{example}

In a coCartesian restriction category with restriction zeroes, the initial object becomes a zero object and we can define ``quasi-projections'' for the restriction coproduct (which make the injections into restriction/partial isomorphisms \cite[Sec 2.3.1]{cockett2002restriction}).

\begin{definition} In a coCartesian restriction $\mathbb{X}$ with restriction zeroes, define the maps $\iota^\circ_j: A_0 \oplus \hdots \oplus A_n \to A_j$ as the copairing $\iota^\circ_j := [0, \hdots, 0, 1_{A_j}, 0, \hdots, 0]$.
\end{definition}

\begin{lemma}\label{lem:restcoprodzero} \cite[Lemma 2.10]{cockett2007restriction} In a coCartesian restriction $\mathbb{X}$ with restriction zeroes,
\begin{enumerate}[{\em (i)}]
\item \label{lem:restcoprodzero.i} The restriction initial object $\mathsf{0}$ is a zero object, so in particular a terminal object;
\item\label{lem:restcoprodzero.ii} $\overline{\iota^\circ_j} = 0 \oplus \hdots \oplus 0 \oplus 1_{A_j} \oplus 0 \oplus \hdots \oplus 0$;
\item \label{lem:restcoprodzero.iii} $\iota_j \iota^\circ_j =  1_{A_j}$ and $\iota_i \iota^\circ_j =0$ for $i \neq j$;
\item \label{lem:restcoprodzero.iv} $\iota^\circ_j \iota_j = \overline{\iota^\circ_j}$;
\item \label{lem:restcoprodzero.v} $\overline{\iota^\circ_j}\overline{\iota^\circ_j}= \overline{\iota^\circ_j}$ and $\overline{\iota^\circ_i}\overline{\iota^\circ_j}= 0$ for $i \neq j$;
\item $\iota^\circ\iota_0 \oplus \hdots \oplus \iota^\circ_n \iota_n = 1_{A_0 \oplus \hdots \oplus A_n}$. 
\end{enumerate}
\end{lemma}

We now use restriction zero maps to define the projections for our desired classical product. 

\begin{definition}\label{def:classprod} In a distributive restriction category $\mathbb{X}$,
\begin{enumerate}[{\em (i)}]
\item For every pair of objects $A$ and $B$, define the object $A \& B$ as $A \& B :=A \oplus  B \oplus  (A \times B)$;
\item If $\mathbb{X}$ has restriction zeroes, then for every pair of objects $A$ and $B$, define the maps ${p_0: A \& B \to A}$ and ${p_1: A \& B \to B}$ respectively as follows: 
 \begin{equation}\begin{gathered} \label{def:p_i}
  \xymatrixcolsep{7pc}\xymatrix{ A & \ar[l]_-{p_0 := [1_A, 0 , \pi_0]} A \oplus  B \oplus  (A \times B) \ar[r]^-{p_1 := [0,1_B, \pi_1]} & B }  \end{gathered}\end{equation}  
\end{enumerate}
A distributive restriction category $\mathbb{X}$ is said to have \textbf{classical products} if $\mathbb{X}$ has restriction zeroes and for every pair of objects $A$ and $B$, (\ref{def:p_i}) is a product diagram, that is, $A \& B$ is the product of $A$ and $B$ with projections $p_0: A \& B \to A$ and $p_1: A \& B \to B$. 
\end{definition}

To distinguish between classical products and restriction products, the maps $p_i$ will now be referred to as the \textbf{classical projections}, and $\llangle -, - \rrangle$ as the \textbf{classical pairing}. It follows from Lemma \ref{lem:restcoprodzero}.(\ref{lem:restcoprodzero.i}) that having classical products implies having finite products: 

\begin{corollary} A distributive restriction category with classical products has finite products. 
\end{corollary}

Here are our main examples of classical products: 

\begin{example} \normalfont  $\mathsf{PAR}$ has classical products where $X \& Y = X \sqcup Y \sqcup (X \times Y)$ and the classical projections are defined as: 
\begin{gather*}
   p_0(x) = x \qquad p_1(x) \uparrow \qquad \forall x\in X \\
   p_0(y) \uparrow \qquad p_1(y)=y \qquad \forall y \in Y \\
   p_0(x,y) = x \qquad p_1(x,y) = y \qquad \forall (x,y) \in X \times Y 
\end{gather*}
The classical pairing $\llangle - , - \rrangle$ of partial functions is defined as follows  
\[ \llangle f,g \rrangle (x) = \begin{cases} f(x) & \text{if } f(x) \downarrow \text{ and } g(x) \uparrow \\
g(x) & \text{if } f(x) \uparrow \text{ and } g(x) \downarrow \\
(f(x), g(x)) & \text{if } f(x) \downarrow \text{ and } g(x) \downarrow \\
\uparrow & \text{if } f(x) \uparrow \text{ and } g(x) \uparrow  \end{cases} \]
\end{example}

\begin{example} \normalfont  $k\text{-}\mathsf{CALG}^{op}_\bullet$ has classical products, so $k\text{-}\mathsf{CALG}_\bullet$ has coclassical coproducts where $A \& B = A \times B \times (A \otimes B)$, and the coclassical injections $p_0: A \to A \&B$ and ${p_1: B \to A\& B}$ are defined as: 
\begin{align*}
    p_0(a) = (a, 0, a \otimes 1) && p_1(b) = (0,b, 1 \otimes b)
\end{align*}
The coclassical copairing $\llangle f,g \rrangle: A \& B \to C$ is defined as follows: 
\begin{align*}
    \llangle f,g \rrangle(a,b, x \otimes y) = f(a) - f(a)g(1) + g(b) - f(1)g(b) + f(x)g(y)  
\end{align*}
\end{example}

It is important to note that for classical products, while on objects we have that ${A \& B = A \oplus B \oplus (A \times B)}$, on maps $f: A \to B$ and $g: C \to D$, the map $f \& g = \llangle p_0 f, p_1 g \rrangle$ may not be equal to $f \oplus g \oplus (f \times g) = \left [ \iota_0 f, \iota_1 g, \iota_2 \langle \pi_0 f, \pi_1 g \rangle \right]$, even though they are of the same type $A \oplus B \oplus (A \times B) \to C \oplus D \oplus (C \times D)$. Intuitively, the main difference between the two is that for $f \oplus g \oplus (f \times g)$, the third component $C \times D$ of its output in only depends on the third component $A \times B$ of the input, while for $f \& g$, the third component $C \times D$ of its output can depend on all three components $A$, $B$, and $A \times B$ of the input. To help better understand this difference, let us compare them explicitly in our main examples. 

\begin{example} \normalfont In $\mathsf{PAR}$, for partial functions $f: X \to Z$ and $g: Y \to W$, let us compare $f \& g$ and $f \sqcup g \sqcup (f \times g)$, both of which are of type $X \sqcup Y \sqcup (X \times Y) \to Z \sqcup W \sqcup (Z \times W)$. On the one hand, for $u \in X \sqcup Y \sqcup (X \times Y)$, we have that: 
\[ \left( f \sqcup g \sqcup (f \times g) \right) (u) = \begin{cases} f(u) & \text{if } u\in X  \text{ and } f(u) \downarrow \\
g(u) & \text{if } u\in Y \text{ and } g(u) \downarrow \\
(f(u_0), g(u_1)) & \text{if } u = (u_0, u_1) \in X \times Y \text{ and } f(u_0) \downarrow \text{ and } g(u_1) \downarrow \\
\uparrow & \text{o.w.}  \end{cases} \]
while on the other hand, we have that: 
\[ \left(f \& g \right) (u) = \begin{cases} f(u) & \text{if } u\in X  \text{ and } f(u) \downarrow \\
g(u) & \text{if } u\in Y \text{ and } g(u) \downarrow \\
(f(u_0), g(u_1)) & \text{if } u = (u_0, u_1) \in X \times Y \text{ and } f(x) \downarrow \text{ and } g(x) \downarrow \\
f(u_0) & \text{if } u = (u_0, u_1) \in X \times Y \text{ and } f(u_0) \downarrow \text{ and } g(u_1) \uparrow \\
g(u_1) & \text{if } u = (u_0, u_1) \in X \times Y \text{ and } f(u_0) \uparrow \text{ and } g(u_1) \downarrow \\
\uparrow & \text{o.w. } \end{cases} \]
As such, we clearly see that $f \& g$ and $f \sqcup g \sqcup (f \times g)$ are indeed different. The main difference between the two is that for $(x, y) \in X \times Y$, we have that $\left( f \sqcup g \sqcup (f \times g) \right)(x,y) = (f(x), g(y)) \in Z \times W$ always (when defined), while $(f \& g)(x,y)$ can equal $f(x) \in Z$, or $g(y) \in W$, or even $(f(x), g(y)) \in Z \times W$. As a particular case, consider when $g=0$. Then $(f \sqcup 0 \sqcup f \times 0)(x) = f(x)$ for all $x \in X$ when $f(x) \downarrow$ but is undefined for all $y \in Y$ and $(x,y) \in X \times Y$, while $(f \& 0)(x)=f(x)$ and $(f \& 0)(x,y)=f(x)$ for all $x \in X$ and $(x,y) \in X \times Y$ when $f(x) \downarrow$ and is undefined for all $y \in Y$. 
\end{example}

\begin{example} \normalfont In $k\text{-}\mathsf{CALG}_\bullet$, for a pair of non-unital $k$-algebra morphism $f: A \to C$ and $g: B \to D$, let us compare $f \&g$ and $f \times g \times (f \otimes g)$, which recall both are of type $A \times B \times (A \otimes B) \to C \times D \times (C \times D)$. On the one hand, we have that: 
\begin{gather*}
\left( f \times g \times (f \otimes g) \right)(a, b, x\otimes y) = (f(a), g(b), f(x) \otimes g(y))
\end{gather*}
while on the other hand we have that: 
\begin{gather*}
(f \& g)(a, b, x\otimes y) \\
= \left(f(a), g(b), f(a) \otimes 1 - f(a) \otimes g(1) + 1 \otimes g(b) - f(1) \otimes g(b) + f(x) \otimes g(y) \right)
\end{gather*}
Thus, we clearly see that the main difference between the two is that the resulting third component of $f \times g \times (f \otimes g)$ only depends on the third input, while the third component of $f \& g$ depends on all three inputs. In particular, taking $g=0$, we see that $\left( f \times 0 \times (f \otimes 0) \right)(a, b, x\otimes y)=(f(a), 0,0)$ while $(f\& g)(a,b,x\otimes y)= (f(a), 0, f(a) \otimes 1)$. 
\end{example}

We conclude this section with some useful identities regarding classical products, which we will need for the main results of this paper in Sec \ref{sec:CDRC}.  

\begin{lemma}\label{lem:cprod} In a distributive restriction category $\mathbb{X}$ with classical products, 
\begin{enumerate}[{\em (i)}]
\item \label{lem:cprod.i} $\overline{p_0} = 1_A \oplus  0 \oplus  1_{A \times B}$ and $\overline{p_1} = 0 \oplus  1_B \oplus  1_{A \times B}$; 
\item\label{lem:cprod.ii} $\overline{p_1}p_0 = \iota^\circ_2 \pi_0$ and $\overline{p_0}p_1 = \iota^\circ_2 \pi_1$; 
\item \label{lem:cprod.iii} $\llangle f, g \rrangle \iota_2^\circ = \langle f,g \rangle$;
\item \label{lem:cprod.iv}  $\llangle f,0 \rrangle = f \iota_0$ and $\llangle 0,f \rrangle = f \iota_1$;
\item \label{lem:cprod.v}  $\llangle f,f \rrangle = \langle f, f \rangle \iota_2$;
\item \label{lem:cprod.vi} $\llangle 0,0 \rrangle = 0$ 
\end{enumerate}
\end{lemma}
\begin{proof}  First recall that the copairing is compatible with composition in the following sense:
\begin{align}\label{eq:copair-comp}
    (g_0 \oplus \hdots \oplus g_n)[f_0, \hdots, f_n] h = [ g_0 f_0 h, \hdots, g_n f_n h]
\end{align}
Also recall that the restriction of a copairing is the coproduct of the restrictions \cite[Lemma 2.1]{cockett2007restriction}: 
\begin{align}\label{eq:rest-coprod}
    \overline{[f_0, \hdots, f_n]} = \overline{f_0} \oplus \hdots \oplus \overline{f_n}
\end{align}
\begin{enumerate}[{\em (i)}]
\item Taking the restriction of the copairing, we compute that: 
\begin{align*}
    \overline{p_0} &=~ \overline{\left[1_A, 0, \pi_0 \right]} \tag*{Def. of $p_0$} \\
    &=~ \overline{1_A} \oplus \overline{0} \oplus \overline{\pi_0} \tag*{(\ref{eq:rest-coprod})}\\
    &=~ 1_A \oplus 0 \oplus 1_{A \times B} \tag*{$1_A$ and $\pi_0$ are total, and rest. zero}
\end{align*}
So $\overline{p_0} = 1_A \oplus  0 \oplus  1_{A \times B}$, and similarly we can show that $\overline{p_1} = 0 \oplus  1_B \oplus  1_{A \times B}$. 
\item This follows from (\ref{lem:cprod.i}) above: 
\begin{align*}
    \overline{p_1}p_0 &=~ (0 \oplus  1_B \oplus  1_{A \times B}) \left[1_A, 0, \pi_0 \right] \tag*{Lemma \ref{lem:cprod}.(\ref{lem:cprod.i}) and Def. of $p_0$} \\
    &=~  \left[0, 0, 1_{A \times B} \right] \pi_0 \tag*{(\ref{eq:copair-comp})} \\
    &=~ \iota^\circ_2 \pi_0 \tag*{Def. of $\iota^\circ_2$}
\end{align*}
So $\overline{p_1}p_0 = \iota^\circ_2 \pi_0$, and similarly we can show that $\overline{p_0}p_1 = \iota^\circ_2 \pi_1$. 
\item We use (\ref{lem:cprod.ii}) above to show that $\llangle f, g \rrangle \iota_2^\circ$ satisfies (\ref{diag:pair1}). So we compute:
\begin{align*}
    \llangle f, g \rrangle \iota_2^\circ \pi_0 &=~  \llangle f, g \rrangle \overline{p_1}p_0 \tag*{Lemma \ref{lem:cprod}.(\ref{lem:cprod.ii})} \\
    &=~ \left\llangle \overline{g}f, \overline{f}g \right\rrangle p_0 \tag*{Lemma \ref{lem:rest-product}.(\ref{lem:rest-product.ii})} \\
    &=~ \overline{g}f \tag*{(\ref{diag:pair2})}
\end{align*}
So $\llangle f, g \rrangle \iota_2^\circ \pi_0 = \overline{g}f$, and similarly we can show that $\llangle f, g \rrangle \iota_2^\circ \pi_1 = \overline{f}g$. So by the universal property of the restriction product, it follows that $\llangle f, g \rrangle \iota_2^\circ = \langle f, g \rangle$. 
\item By definition of the classical projections, we have that $f \iota_0 p_0=f$ and $f \iota_0 p_1 =0$. Then by the universal property of the classical product, it follows that $\llangle f,0 \rrangle = f \iota_0$. Similarly, we also have that $\llangle 0,f \rrangle = f \iota_1$. 
\item By definition of the classical projections, the pairing $\langle-,- \rangle$, and \textbf{[R.1]}, we have that $\langle f, f \rangle \iota_2 p_0=f$ and $\langle f, f \rangle \iota_2 p_1=f$. Then by the universal property of the classical product, it follows that $\llangle f,f \rrangle = \langle f, f \rangle \iota_2$. 
\item This follows from (\ref{lem:cprod.iv}), so $\llangle 0,0 \rrangle = 0$. 
\end{enumerate}
\end{proof}

\section{Classical Restriction Categories}\label{sec:classical}

The main objective of this paper is to prove that a distributive restriction category with classical products is classical. In this section, we review the basics of classical restriction categories, as well as provide some new results regarding complements of restriction idempotents. For a more in-depth introduction to classical restriction categories, we refer the reader to \cite{cockett2009boolean,cockett2012differential}. 

A restriction category is classical if we can take the joins of maps and relative complements of maps. In order to properly define joins, we first need to discuss certain relations between maps in a restriction category. 

\begin{definition}\label{def:leq} In a restriction category $\mathbb{X}$, for parallel maps $f: A \to B$ and $g: A \to B$ we say that: 
\begin{enumerate}[{\em (i)}]
\item $f$ is \textbf{less than or equal to} $g$ \cite[Sec 2.1.4]{cockett2002restriction}, written $f \leq g$, if $\overline{f}g= f$;
\item $f$ and $g$ are \textbf{compatible} \cite[Prop 6.3]{cockett2009boolean}, written $f \smile g$, if $\overline{f}g = \overline{g}f$; 
\end{enumerate}
If $\mathbb{X}$ has restriction zeroes, then we say that: 
\begin{enumerate}[{\em (i)}]
\setcounter{enumi}{2}
\item $f$ and $g$ are \textbf{disjoint} \cite[Prop 6.2]{cockett2009boolean}, written $f \perp g$, if $\overline{f}g = 0$ (or equivalently $\overline{g}f =0)$. 
\end{enumerate}
\end{definition}

Intuitively, $f \leq g$ means that whenever $f$ is defined, $g$ is equal to $f$; $f \smile g$ means that whenever they are both defined, they are equal; and $f \perp g$, means that whenever one is defined, the other is not. Note that disjoint maps are compatible, so if $ f\perp g$ then $f \smile g$ as well. We can also consider the join of compatible maps with respect to the partial order. 

\begin{definition}\label{def:join} In a restriction category $\mathbb{X}$, the \textbf{join} \cite[Def 6.7]{cockett2009boolean} (if it exists) of a finite family of parallel maps $f_0: A \to B$, ..., $f_n : A\to B$ that is pairwise compatible, so $f_i \smile f_j$ for all $0 \leq i,j, \leq n$, is a (necessarily unique) map $f_0 \vee \hdots \vee f_n: A \to B$ such that:
\begin{enumerate}[{\bf [J.1]}]
\item $f_i \leq f_0 \vee \hdots \vee f_n$ for all $0 \leq i \leq n$;
\item If $g: A \to B$ is a map such that $f_i \leq g$ for all $0 \leq i \leq n$ then $f_0 \vee \hdots \vee f_n \leq g$;
\end{enumerate}
A \textbf{join restriction category} \cite[Def 10.1]{cockett2009boolean} is a restriction category $\mathbb{X}$ such that the join of any finite family of pairwise compatible maps exists and is preserved by pre-composition, that is: 
\begin{enumerate}[{\bf [J.1]}]
\setcounter{enumi}{2}
\item For any map $h: A^\prime \to A$, $h\left(f_0 \vee \hdots \vee f_n \right) = hf_0 \vee \hdots \vee hf_n$.
\end{enumerate}
\end{definition}

Here are now some useful identities regarding joins. In particular, join restriction categories have restriction zeroes, given by the join of the empty family, and joins are also preserved by post-composition. 

\begin{lemma}\label{lem:join}\cite[Prop 2.14]{cockett2012differential} In a join restriction category $\mathbb{X}$,
\begin{enumerate}[{\em (i)}]
\item \label{lem:join.i} $\mathbb{X}$ has restriction zeroes;
\item \label{lem:join.ii} $f \vee 0 = f$; 
\item \label{lem:join.iii} $\left(f_0 \vee \hdots \vee f_n \right)k = f_0k \vee \hdots \vee f_nk$
\item \label{lem:join.iv} $\overline{f_0 \vee \hdots \vee f_n} = \overline{f_0} \vee \hdots \vee \overline{f_n}$
\item \label{lem:join.v} $\overline{f_j}\left( f_0 \vee \hdots \vee f_n \right) = f_j$
\item \label{lem:join.vi} Restriction idempotents $e_i$ are always pairwise compatible and $e_0 \vee \hdots \vee e_n$ is a restriction idempotent;
\item \label{lem:join.vii} If $e_1$ and $e_2$ are restriction idempotents, then $e_1 \perp e_2$ if and only if $e_1 e_2 =0$. 
\end{enumerate}
\end{lemma}

Here are also some compatibility identities between joins and restriction coproducts which will be useful for proofs in later sections. 

\begin{lemma}\label{lem:class-coprod} In a coCartesian restriction category $\mathbb{X}$ which is also a join restriction category, 
    \begin{enumerate}[{\em (i)}]
\item \label{lem:class-coprod.ii}  For any family of maps $f_0$, ..., and $f_n$, the maps $\iota^\circ_0 f_0 \iota_0$, ..., and $\iota^\circ_n f_n \iota_n$ are pairwise disjoint, and $\iota^\circ_0 f_0 \iota_0 \vee \hdots \vee \iota^\circ_n f_n \iota_n = f_0 \oplus \hdots \oplus f_n$; 
\item \label{lem:class-coprod.iv} $\iota^\circ_0 \iota_0$, ..., and $\iota^\circ_n \iota_n$ are pairwise disjoint, and $\iota^\circ_0 \iota_0 \vee \hdots \vee \iota^\circ_n \iota_n = 1_{A_0 \oplus \hdots \oplus A_n}$. 
\end{enumerate}
\end{lemma}
\begin{proof} 
   \begin{enumerate}[{\em (i)}]
\item For $i \neq j$, we compute: 
\begin{align*}
    \overline{\iota^\circ_i f_i \iota_i} \iota^\circ_j f_j \iota_j &=~ \overline{\iota^\circ_i f_i \iota_i} \overline{\iota^\circ_j}\iota^\circ_j f_j \iota_j \tag*{\textbf{[R.1]}} \\
    &=~ \overline{\iota^\circ_j} \overline{\iota^\circ_i f_i \iota_i} \iota^\circ_j f_j \iota_j \tag*{\textbf{[R.2]}} \\
    &=~  \overline{\overline{\iota^\circ_j} \iota^\circ_i f_i \iota_i} \iota^\circ_j f_j \iota_j\tag*{\textbf{[R.3]}} \\
    &=~ \overline{0} \iota^\circ_j f_j \iota_j \tag*{Lemma \ref{lem:restcoprodzero}.(\ref{lem:restcoprodzero.v})} \\ 
    &=~ 0 \tag*{Rest. zero}
    \end{align*}
So we have that $\iota^\circ_i f_i \iota_i \perp \iota^\circ_j f_j \iota_j$ for $i \neq j$. Thus we can take their join $\iota^\circ_0 f_0 \iota_0 \vee \hdots \vee \iota^\circ_n f_n \iota_n$. Precomposing with the injection maps, we compute:
\begin{align*}
    \iota_j \left( \iota^\circ_0 f_0 \iota_0 \vee \hdots \vee \iota^\circ_n f_n \iota_n \right) &=~ \iota_j\iota^\circ_0 f_0 \iota_0 \vee \hdots \vee \iota_j\iota^\circ_n f_n \iota_n \tag*{\textbf{[J.3]}} \\
    &=~ 0 \vee \hdots \vee 0 \vee f_j \iota_j \vee 0 \vee \hdots \vee 0 \tag*{Lemma \ref{lem:restcoprodzero}.(\ref{lem:restcoprodzero.iii})} \\ 
    &=~ f_j \iota_j \tag*{Lemma \ref{lem:join}.(\ref{lem:join.ii})} 
\end{align*}
So by the couniversal property of the coproduct, $\iota^\circ_0 f_0 \iota_0 \vee \hdots \vee \iota^\circ_n f_n \iota_n = [f_0 \iota_0, \hdots, f_n \iota_n]$, which we can alternatively write as $\iota^\circ_0 f_0 \iota_0 \vee \hdots \vee \iota^\circ_n f_n \iota_n = f_0 \oplus \hdots \oplus f_n$. 
    \item This is a special case of (\ref{lem:class-coprod.ii}) by setting $f_j = 1_{A_j}$. 
\end{enumerate}
\end{proof}

We now turn our attention to classical restriction categories, which are join restriction categories that also have relative complements of maps. 

\begin{definition}\label{def:classrest} In a join restriction category $\mathbb{X}$, for parallel maps $f: A \to B$ and $g: A \to B$ such that $f \leq g$, the \textbf{relative complement} \cite[Sec 13]{cockett2009boolean} of $f$ in $g$ (if it exists) is a (necessarily unique) map $g \backslash f: A \to B$ such that\footnote{We note that in \cite{cockett2009boolean}, the axioms were written as (1) $g \backslash f \leq g$, (2) $g \backslash f \perp f$, and (3) $g \leq g \backslash f \vee f$. However, it is straightforward to check that \textbf{[RC.2]} is an equivalent way of expressing (1) and (3).}: 
\begin{enumerate}[{\bf [RC.1]}]
\item $g \backslash f \perp f$
\item $g \backslash f \vee f =g$
\end{enumerate}
A \textbf{classical restriction category} \cite[Prop 13.1]{cockett2009boolean} is a join restriction category $\mathbb{X}$ such that all relative complements exists 
\end{definition}

Intuitively, $g \backslash f$ is undefined when $f$ is defined and is equal to $g$ whenever $f$ is undefined. Here are our main examples of classical restriction categories. 

\begin{example} \normalfont  $\mathsf{PAR}$ is a classical restriction category where: 
\begin{enumerate}[{\em (i)}]
\item $f \leq g$ if $g(x)=f(x)$ whenever $f(x) \downarrow$.
\item $f \smile g$ if $f(x)=g(x)$ whenever both $f(x) \downarrow$ and $g(x) \downarrow$.
\item $f \perp g$ if $f(x) \uparrow$ whenever $g(x) \downarrow$, and $g(x) \uparrow$ whenever $f(x) \downarrow$.
\item If $f \smile g$, then their join $f \vee g$ is defined as follows: 
\[ (f \vee g) (x) = \begin{cases} f(x) & \text{if } f(x) \downarrow \text{ and } g(x) \uparrow \\
g(x) & \text{if } f(x) \uparrow \text{ and } g(x) \downarrow \\
f(x)=g(x) & \text{if } f(x) \downarrow \text{ and } g(x) \downarrow \\
\uparrow & \text{if } f(x) \uparrow \text{ and } g(x) \uparrow  \end{cases} \]
\item If $f \leq g$, then the relative complement $g \backslash f$ is defined as follows: 
\[ (g \backslash f) (x) = \begin{cases} g(x) & \text{if } f(x) \uparrow \text{ and } g(x) \downarrow \\
\uparrow & \text{if } f(x) \downarrow \text{ or } g(x) \uparrow  \end{cases} \]
\end{enumerate}
\end{example}

\begin{example} \normalfont  $k\text{-}\mathsf{CALG}^{op}_\bullet$ is a classical restriction category, so $k\text{-}\mathsf{CALG}_\bullet$ is a coclassical corestriction category where: 
\begin{enumerate}[{\em (i)}]
\item $f \leq g$ if $f(1)g(a) = f(a)$. 
\item $f \smile g$ if $g(1)f(a) = f(1) g(a)$.
\item $f \perp g$ if $g(1)f(a) = 0 = f(1) g(a)$. 
\item If $f \smile g$, then their join $f \vee g$ is defined as follows: 
\[ (f \vee g)(a) = f(a) + g(a) - g(1)f(a) = f(a) + g(a) - f(1) g(a) \]
\item If $f \leq g$, then the relative complement $g \backslash f$ is defined as follows: 
\[ (g \backslash f) (a) = g(a) - f(1)g(a)   \]
\end{enumerate}
That $k\text{-}\mathsf{CALG}^{op}_\bullet$ is a classical restriction category is a new observation that follows from the main results of this paper. 
\end{example}

We conclude this section by discussing complements of restriction idempotents. Indeed, note that a restriction idempotent $e: A \to A$ is always less than or equal to the identity $1_A: A \to A$, so $e \leq 1_A$. As such, we can consider the relative complement of a restriction idempotent in the identity. 

\begin{definition}\label{def:ecomp}  In a classical restriction category $\mathbb{X}$, if ${e: A \to A}$ is a restriction idempotent, we denote the relative complement of $e$ in $1_A$ by $e^c = 1_A \backslash e$. We call $e^c$ the \textbf{complement} of $e$. 
\end{definition}

Complements of restriction idempotents will play a crucial role in proving that a classical distributive restriction category has classical products. Indeed, we will be able to define the classical pairing $\llangle -,- \rrangle$ using joins and complements of restrictions. Here are now some useful identities regarding complements of restriction idempotents. 

\begin{lemma} \label{lem:ecomp}  In a classical restriction category $\mathbb{X}$, 
    \begin{enumerate}[{\em (i)}]
\item \label{lem:ecomp.0} If $e$ is a restriction idempotent, then $e^c$ is a restriction idempotent and ${e^c}^c=e$; 
\item \label{lem:ecomp.ii} If $e$ is a restriction idempotent, then $e e^c = 0$ and $e \vee e^c = 1$; 
\item \label{lem:ecomp.iv} If $e^\prime$ is a restriction idempotent such that $ee^\prime=0$ and $e \vee e^\prime = 1$, then $e^c = e^\prime$;
\item \label{lem:ecomp.vii} If $e_1$ and $e_2$ are restriction idempotents, then $e^c_1 \vee e^c_2 = (e_1 e_2)^c$ and $e^c_1 e_2^c =(e_1 \vee e_2)^c$;
\item \label{lem:ecomp.viii} $1^c = 0$ and $0^c = 1$; 
\item \label{lem:ecomp.i} $\overline{f}^c f = 0$;
\item \label{lem:ecomp.ix} $\overline{f}^c \overline{g}^c = \overline{g}^c \overline{f}^c$;
\item \label{lem:ecomp.x} $\overline{ \overline{g}^c f}^c = \overline{g} \vee \overline{f}^c$; 
\item  \label{lem:ecomp.vi} $f \overline{g}^c = \overline{fg}^c f$. 
\end{enumerate}
\end{lemma}
\begin{proof} First recall some useful identities about the relative complement \cite[Lemma 13.14]{cockett2009boolean}: 
\begin{gather}
    \overline{g \backslash f} = \overline{g} \backslash \overline{f} \label{lem:rc.i} \\ 
    g \backslash \left( g \backslash f \right) = f \label{lem:rc.ii} \\ 
    (g_1 \backslash f_1) (g_2 \backslash f_2) = g_1 g_2 \backslash \left( g_1 f_2 \vee f_1 g_2 \right) \label{lem:rc.v} \\
    h\left( g \backslash f \right) k = hgk \backslash hfk \label{lem:rc.iii} \\
    f \backslash f =0 \label{lem:rc.iv} \\
    g \backslash 0 = g \label{lem:rc.vi}
\end{gather}
    \begin{enumerate}[{\em (i)}]
\item We compute that: 
\begin{align*}
    \overline{e^c} &=~ \overline{1_A \backslash e} \tag*{Def. of $e^c$} \\
    &=~ \overline{1_A} \backslash \overline{e} \tag*{(\ref{lem:rc.i})} \\
    &=~ 1_A \backslash e \tag*{$1_A$ and $e$ rest. idemp.} \\
    &=~ e^c \tag*{Def. of $e^c$}
\end{align*}
So $\overline{e^c} = e^c$, and thus $e^c$ is a restriction idempotent. Taking its complement we get: 
\begin{align*}
    {e^c}^c &=~ 1_A \backslash \left( 1_A \backslash e \right) \tag*{Def. of $e^c$} \\
    &=~ e \tag*{(\ref{lem:rc.i})} 
\end{align*}
So ${e^c}^c=e$, as desired. 
\item We first compute that: 
\begin{align*}
   e \overline{e^c} &=~e  \overline{1_A \backslash e} \tag*{Def. of $e^c$} \\
    &=~ e \backslash (ee) \tag*{(\ref{lem:rc.iii})} \\
    &=~ e \backslash e \tag*{Lemma \ref{lem:rest}.(\ref{lem:rest.idem.R1})} \\
    &=~ 0 \tag*{(\ref{lem:rc.iv})}
\end{align*}
So $ e \overline{e^c} =0$. On the other hand, $e \vee e^c = 1$ is simply \textbf{[RC.2]}. 
\item Since $e$ and $e^\prime$ are both restriction idempotents, we have that $e, e^\prime \leq 1$. Now $e e^\prime =0$ implies that $e \perp e^\prime$, which is \textbf{[RC.1]}, while $e \vee e^\prime = 1$ is \textbf{[RC.2]}. Therefore by the uniqueness of relative complements, we have that $e^\prime = 1 \backslash e$. So $e^\prime = e^c$. 
\item These are the de Morgan laws. We first compute that: 
\begin{align*}
e^c_1 e^c_2 &=~ (1\backslash e_1) (1\backslash e_2) \tag*{Def. of $e^c_i$} \\
&=~ 1 \backslash (e_1 \vee e_2) \tag*{(\ref{lem:rc.v})} \\
&=~ (e_1 \vee e_2)^c \tag*{Def. of $(e_1 \vee e_2)^c$}
\end{align*}
So $e^c_1 e^c_2 = (e_1 \vee e_2)^c$. Then applying (\ref{lem:ecomp.0}), we also get $e^c_1 \vee e^c_2 = (e_1 e_2)^c$. 
\item By (\ref{lem:rc.iv}), $1^c = 0$, and by (\ref{lem:rc.vi}), $0^c =1$. 
\item This is the analogue of \textbf{[R.1]} for the complement of the restriction. So we compute: 
\begin{align*}
    \overline{f}^c f &=~ (1_A \backslash \overline{f}) f \tag*{Def. of $\overline{f}^c$} \\
    &=~ f \backslash (\overline{f}f) \tag*{(\ref{lem:rc.iii})} \\
    &=~ f \backslash f \tag*{\textbf{[R.1]}} \\
    &=~ 0 \tag*{(\ref{lem:rc.iv})} 
\end{align*}
So $\overline{f}^c f = 0$. 
\item This is the analogue of \textbf{[R.2]} for the complement of the restriction. Since by (\ref{lem:ecomp.0}), $\overline{f}^c$ and $\overline{g}^c$ are restriction idempotents, and restriction idempotents commute (Lemma \ref{lem:rest}.(\ref{lem:rest.idem.R2})), we have that $\overline{f}^c \overline{g}^c = \overline{g}^c \overline{f}^c$. 
\item This is the analogue of \textbf{[R.3]} for the complement of the restriction. So we compute: 
\begin{align*}
    \overline{ \overline{g}^c f}^c &=~ (\overline{g}^c \overline{f})^c \tag*{Lemma \ref{lem:rest}.(\ref{lem:rest.idem.R3})} \\
    &=~ {\overline{g}^c}^c \vee \overline{f}^c \tag*{Lemma \ref{lem:ecomp}.(\ref{lem:ecomp.vii})} \\
    &=~  \overline{g} \vee \overline{f}^c \tag*{Lemma \ref{lem:ecomp}.(\ref{lem:ecomp.0})} 
\end{align*}
So $\overline{ \overline{g}^c f}^c = \overline{g} \vee \overline{f}^c$. 
\item This is the analogue of \textbf{[R.4]} for the complement of the restriction. So we compute: 
\begin{align*}
    f \overline{g}^c &=~ \overline{f \overline{g}^c} f \tag*{Lemma \ref{lem:rest}.(\ref{lem:rest.idem.R4})} \\
    &=~ \overline{f (1_B \backslash \overline{g})} f \tag*{Def. of $\overline{g}^c$} \\
    &=~ \overline{ f \backslash f\overline{g} } f \tag*{(\ref{lem:rc.iii})} \\
    &=~ \overline{ f \backslash \overline{fg}f } f \tag*{\textbf{[R.4]}} \\
    &=~ \left( \overline{f} \backslash \overline{\overline{fg}f} \right) f \tag*{Lemma \ref{lem:ecomp}.(\ref{lem:rc.i})} \\
    &=~ \left( \overline{f} \backslash \overline{fg} \overline{f} \right)f \tag*{\textbf{[R.3]}} \\
    &=~ \left( 1_A \backslash \overline{fg} \right) \overline{f} f \tag*{(\ref{lem:rc.iii})} \\
    &=~ \left( 1_A \backslash \overline{fg} \right)  f \tag*{\textbf{[R.1]}} \\
    &=~  \overline{fg}^c f \tag*{Def. of $\overline{fg}^c$} \\
\end{align*}
So $f \overline{g}^c = \overline{fg}^c f$
\end{enumerate}
\end{proof}

Here are the complements of the restriction idempotents in our two main examples:

\begin{example} \normalfont In $\mathsf{PAR}$, for a subset $U \subseteq X$, the complement of its corresponding restriction idempotent $e_U$ is the restriction idempotent associated to the complement of $U$, $U^c = \lbrace x \in X \vert x \notin X \rbrace \subseteq X$, so $e^c_U = e_{U^c}$. As such we have that: 
\[ e^c_U(x) = \begin{cases} x & \text{if } x \notin U \\
\uparrow & \text{if } x \in U \end{cases} \]
So in particular, for a partial function $f: X \to Y$, the complement of its restriction $\overline{f}$ is given by:
\[ \overline{f}^c(x) = \begin{cases} x & \text{if } f(x) \uparrow \\
\uparrow & \text{if } f(x) \downarrow \end{cases} \]
\end{example}

\begin{example} \normalfont In $k\text{-}\mathsf{CALG}_\bullet$, for an idempotent element $u \in A$, the complement of its associated corestriction idempotent $e_u$ is the corestriction idempotent associated to the idempotent $1-u$, that is, $e^c_u = e_{1-u}$ and so: 
\[ e^c_u(a) = a -ua  \]
So for a non-unital $k$-algebra morphism $f: A \to B$, the complement of its corestriction $\overline{f}$ is given by: 
\[ \overline{f}^c(b) = b - f(1)b \]
\end{example}

\section{From Classical Products to Restriction Products and Coproducts}\label{sec:split}

In this section, we will show how classical products are related to restriction products and restriction coproducts via restriction idempotent splitting. Furthermore, using split restriction idempotents, we will also be able to show that in a classical restriction category, if the product exists, then it must always be of the form $A \& B = A \oplus B \oplus (A \times B)$. 

\begin{definition}  In a restriction category $\mathbb{X}$, a restriction idempotent $e: A \to A$ is said to be \textbf{split} \cite[Sec 2.3.3]{cockett2002restriction} if there exists maps $r: A \to X$ and $s: X \to A$ such that $rs =e$ and $sr=1_X$. 
\end{definition}

Suppose we are in a restriction category with products. By Lemma \ref{lem:rest}.(\ref{lem:rest.idem.R2}), we have that ${\overline{p_0}~\overline{p_1}: A \& B \to A \& B}$ is a restriction idempotent. The splitting of this restriction idempotent relates the product $A \& B$ to the restriction product $A \times B$. 

\begin{lemma} \label{lem:prod-restprod} Let $\mathbb{X}$ be a restriction category with binary products. Then $\mathbb{X}$ has binary restriction products if and only if for every pair of objects $A$ and $B$, the restriction idempotent $\overline{p_0}~\overline{p_1}: A \& B \to A \& B$ splits. Explicitly: 
\begin{enumerate}[{\em (i)}]
\item If the restriction product $A \times B$ exists, then $\overline{p_0}~\overline{p_1}$ splits via $\langle p_0, p_1 \rangle: A \& B \to A \times B$ and $\left\llangle \pi_0, \pi_1 \right\rrangle: A \times B \to A\&B$. 
\item If $\overline{p_0}~\overline{p_1}$ splits via $r: A \& B \to A \times B$ and $s: A \times B \to A \& B$, for some object $A \times B$, then $A \times B$ is a restriction product of $A$ and $B$ with projections $\pi_0: A \times B \to A$ and $\pi_1: A \times B \to B$ defined as $\pi_0 = sp_0$ and $\pi_1 = sp_1$. 
\end{enumerate}
\end{lemma}
\begin{proof} First recall some useful basic identities about the restriction product \cite[Prop 2.8]{cockett2012differential}:
\begin{gather}
    k\langle f,g \rangle = \langle kf, kg \rangle \label{eq:rest-prod1} \\
    \langle \pi_0, \pi_1 \rangle = 1_{A \times B}  \label{eq:rest-prod2}
\end{gather}
\begin{enumerate}[{\em (i)}]
\item Suppose that $A \times B$ is a restriction product of $A$ and $B$. We must show that $\langle p_0, p_1 \rangle  \left\llangle \pi_0, \pi_1 \right\rrangle = \overline{p_0}~\overline{p_1}$ and $\left\llangle \pi_0, \pi_1 \right\rrangle \langle p_0, p_1 \rangle = 1_{A \times B}$. So we compute: 
\begin{align*}
    \langle p_0, p_1 \rangle  \left\llangle \pi_0, \pi_1 \right\rrangle &=~ \left\llangle \langle p_0, p_1 \rangle \pi_0, \langle p_0, p_1 \rangle \pi_1 \right\rrangle \tag*{(\ref{eq:pair-comp})} \\ 
    &=~ \left\llangle \overline{p_1} p_0, \overline{p_0} p_1 \right\rrangle \tag*{(\ref{diag:pair1})} \\
&=~ \left\llangle p_0,  p_1 \right\rrangle  \overline{p_0}~\overline{p_1}  \tag*{Lemma \ref{lem:rest-product}.(\ref{lem:rest-product.iii})} \\
&=~ \overline{p_0}~\overline{p_1} \tag*{$\left\llangle p_0,  p_1 \right\rrangle = 1_{A \& B}$}
\end{align*}
\begin{align*}
    \left\llangle \pi_0, \pi_1 \right\rrangle \langle p_0, p_1 \rangle &=~ \left \langle   \left\llangle \pi_0, \pi_1 \right\rrangle p_0,  \left\llangle \pi_0, \pi_1 \right\rrangle p_1 \right \rangle \tag*{(\ref{eq:rest-prod1})} \\ 
    &=~ \left \langle   \pi_0,  \pi_1 \right \rangle \tag*{(\ref{diag:pair2})} \\ 
    &=~ 1_{A \times B} \tag*{(\ref{eq:rest-prod2})}
\end{align*}
So we conclude that $\overline{p_0}~\overline{p_1}$ splits. 

\item Suppose that $\overline{p_0}~\overline{p_1}$ splits via ${r: A\&B \to A\times B}$ and $s: A\times B \to A\& B$, so: 
\begin{align}\label{ppsplit}
    rs = \overline{p_0}~\overline{p_1} && sr = 1_{A \times B}
\end{align}
We must first show that the suggested projections are total. First note that from (\ref{ppsplit}) it follows that:
\begin{align}\label{eq:spi}
    s \overline{p_i}=s
\end{align}
Furthermore, since $s$ is a section, $s$ is monic, and thus $s$ is total (Lemma \ref{lem:rest}.(\ref{lem:rest.total})). So we compute:
\begin{align*}
 \overline{\pi_i} &=~ \overline{sp_i} \tag*{Def. of $\pi_i$} \\ 
 &=~ \overline{s \overline{p_i}} \tag*{Lemma \ref{lem:rest}.(\ref{lem:rest.i})} \\
 &=~ \overline{s} \tag*{(\ref{eq:spi})} \\
 &=~ 1_{A \times B} \tag*{$s$ total}
\end{align*}
So $\pi_0$ and $\pi_1$ are total as needed. Now given a pair of maps $f: C \to A$ and $g: C \to B$, define $\langle f,g \rangle: C \to A \times B$ as the following composite: 
\begin{align}
  \langle f, g \rangle =  \left\llangle f,g\right\rrangle r  
\end{align}
Then we compute: 
\begin{align*}
     \langle f, g \rangle \pi_0 &=~  \left\llangle f,g\right\rrangle r s p_0 \tag*{Def. of $\langle -,-\rangle$} \\
&=~ \left\llangle f,g\right\rrangle  \overline{p_0}~\overline{p_1} p_0 \tag{\ref{ppsplit}} \\
&=~ \left\llangle \overline{g}f, \overline{f}g\right\rrangle p_0 \tag*{Lemma \ref{lem:rest-product}.(\ref{lem:rest-product.iii})} \\
&=~ \overline{g}f \tag*{(\ref{diag:pair2})}
\end{align*}
So $\langle f, g \rangle \pi_0 = \overline{g}f$, and similarly we can compute that $\langle f, g \rangle \pi_1 = \overline{f}g$. Lastly for uniqueness, suppose that there was another map $h$ such that $h \pi_0 = \overline{f}g$ and $h \pi_1= \overline{g}f$. Then we compute:
\begin{align*}
    h &=~ h sr \tag{\ref{ppsplit}} \\
    &=~ h s \left \llangle p_0, p_1 \right\rrangle r  \tag*{$\left\llangle p_0,  p_1 \right\rrangle = 1_{A \& B}$} \\
    &=~ \left \llangle hs p_0, hs p_1 \right\rrangle r  \tag*{(\ref{eq:pair-comp})} \\ 
    &=~ \left \llangle h\pi_0, h \pi_1 \right\rrangle r \tag*{Def. of $\pi_i$} \\
    &=~ \left \llangle \overline{g}f, \overline{g}f \right\rrangle r \tag*{Assumption on $h$} \\
    &=~  \left\llangle f,g\right\rrangle \overline{p_0}~\overline{p_1} r \tag*{Lemma \ref{lem:rest-product}.(\ref{lem:rest-product.iii})} \\
    &=~  \left\llangle f,g\right\rrangle r \tag{\ref{ppsplit}} \\
    &=~ \langle f, g \rangle \tag*{Def. of $\langle -,-\rangle$} 
\end{align*}
So $\langle f, g \rangle$ is unique. Thus $A \times B$ is the restriction product of $A$ and $B$. 
\end{enumerate}
\end{proof}

To relate the product to the restriction coproduct, we will need to work in a classical restriction category. By Lemma \ref{lem:ecomp}.({\ref{lem:ecomp.0}}) and Lemma \ref{lem:join}.(\ref{lem:join.vi}), $\overline{p_0}^c \vee \overline{p_1}^c: A \& B \to A \& B$ is a restriction idempotent. The splitting of this restriction idempotent relates the product $A \& B$ to the restriction coproduct $A \oplus B$. Before showing this, here are useful identities we will need for the proof. 

\begin{lemma} \label{lem:class-prod} In a classical restriction category $\mathbb{X}$ with binary products, 
\begin{enumerate}[{\em (i)}]
\item \label{lem:class-prod.iii} $\llangle f,g \rrangle \overline{p_0}^c = \llangle 0, \overline{f}^c g \rrangle$ and $\llangle f,g \rrangle \overline{p_1}^c = \llangle \overline{g}^c f, 0 \rrangle$;
\item \label{lem:class-prod.iv} $\llangle f,g \rrangle \overline{p_0}^c p_1 = \overline{f}^c g$ and $\llangle f,g \rrangle \overline{p_1}^c p_0 = \overline{g}^c f$;
\item \label{lem:class-prod.0} $\llangle 0,0 \rrangle =0$;
\item \label{lem:class-prod.i} $\overline{p_0}^c \perp \overline{p_1}^c$, or in other words, $\overline{p_0}^c ~\overline{p_1}^c =0$; 
\item \label{lem:class-prod.vi} $\overline{p_0}^c \vee \overline{p_1}^c = \llangle \overline{p_1}^c p_0, \overline{p_0}^c p_1 \rrangle$;
\item \label{lem:class-prod.ii} $\left\llangle \overline{p_1}^c p_0, 0 \right\rrangle \perp \left\llangle 0, \overline{p_0}^c p_1 \right\rrangle$ and $\left\llangle \overline{p_1}^c p_0, 0 \right\rrangle  \vee  \left\llangle 0, \overline{p_0}^c p_1 \right\rrangle = \overline{p_0}^c \vee \overline{p_1}^c$;
\item \label{lem:class-prod.v}  $\llangle 1_A, 0 \rrangle \overline{p_0}^c=0$, $\llangle 1_A, 0 \rrangle \overline{p_1}^c = \llangle 1_A, 0 \rrangle$, $\llangle 0, 1_B \rrangle \overline{p_0}^c=\llangle 0, 1_B \rrangle$, and $\llangle 0, 1_B \rrangle \overline{p_1}^c = 0$;
\item \label{lem:class-prod.vii} $\overline{p_0}^c$ is a split via the maps $\overline{p_0}^c p_1: A \& B \to B$ and ${\llangle 0,1_B \rrangle: B \to A \& B}$, and $\overline{p_1}^c$ is a split via the maps $\overline{p_1}^c p_0: A\&B \to A$ and $\llangle 1_A, 0  \rrangle: A \to A \& B$. 
\end{enumerate}
\end{lemma}
\begin{proof}
    \begin{enumerate}[{\em (i)}]
    \item We compute: 
\begin{align*}
    \llangle f,g \rrangle \overline{p_0}^c &=~ \overline{\llangle f,g \rrangle  p_0}^c \llangle f,g \rrangle \tag*{Lemma \ref{lem:ecomp}.(\ref{lem:ecomp.vi})} \\
    &=~ \overline{f}^c \llangle f,g \rrangle \tag*{(\ref{diag:pair2})} \\
    &=~ \llangle \overline{f}^c f, \overline{f}^c g \rrangle \tag*{(\ref{eq:pair-comp})} \\
    &=~ \llangle 0, \overline{f}^c g \rrangle \tag*{Lemma \ref{lem:ecomp}.(\ref{lem:ecomp.i})} 
\end{align*}
So $\llangle f,g \rrangle \overline{p_0}^c = \llangle 0, \overline{f}^c g \rrangle$, and similarly we can show that $\llangle f,g \rrangle \overline{p_1}^c = \llangle \overline{g}^c f, 0 \rrangle$.
\item This follows immediatly from (\ref{lem:class-prod.iii}). 
\item Postcomposing zero by the projections results in zero, $0 p_i =0$. So by the universal property of the product, it follows that $\llangle 0,0 \rrangle =0$.
\item By Lemma \ref{lem:ecomp}.(\ref{lem:ecomp.i}), postcomposing $\overline{p_0}^c \overline{p_1}^c$ by the projections gives ${\overline{p_0}^c \overline{p_1}^c p_i = 0}$. Thus by the universal property of the product, it follows that $\overline{p_0}^c \overline{p_1}^c = \llangle 0, 0 \rrangle$. Thus by (\ref{lem:class-prod.0}), $\overline{p_0}^c \overline{p_1}^c =0$, and so $\overline{p_0}^c \perp \overline{p_1}^c$. 
\item We compute: 
\begin{align*}
    \left( \overline{p_0}^c \vee\overline{p_1}^c \right) p_0 &=~ \overline{p_0}^c p_0 \vee \overline{p_1}^c p_0  \tag*{Lemma \ref{lem:join}.(\ref{lem:join.iii})}  \\
    &=~ 0 \vee \overline{p_1}^c p_0  \tag*{Lemma \ref{lem:ecomp}.(\ref{lem:ecomp.i})} \\
    &=~ \overline{p_1}^c p_0  \tag*{Lemma \ref{lem:join}.(\ref{lem:join.ii})} 
\end{align*}
So $\left( \overline{p_0}^c \vee\overline{p_1}^c \right) p_0 = \overline{p_1}^c p_0$, and similarly we can show that $\left( \overline{p_0}^c \vee\overline{p_1}^c \right) p_1 = \overline{p_0}^c p_1$. So by the universal property of the product, it follows that $\overline{p_0}^c \vee \overline{p_1}^c = \llangle \overline{p_1}^c p_0, \overline{p_0}^c p_1 \rrangle$. 
\item To show these maps are disjoint, we compute the following: 
\begin{align*}
    \overline{\left\llangle \overline{p_1}^c p_0, 0 \right\rrangle} \left\llangle 0, \overline{p_0}^c p_1 \right\rrangle p_0 &=~  \overline{\left\llangle \overline{p_1}^c p_0, 0 \right\rrangle} 0 \tag*{(\ref{diag:pair2})} \\
    &=~ 0 
\end{align*}
\begin{align*}
    \overline{\left\llangle \overline{p_1}^c p_0, 0 \right\rrangle} \left\llangle 0, \overline{p_0}^c p_1 \right\rrangle p_1 &=~  \overline{\left\llangle \overline{p_1}^c p_0, 0 \right\rrangle} \overline{p_0}^c p_1 \tag*{(\ref{diag:pair2})} \\
    &=~ \overline{\left\llangle \overline{p_1}^c p_0, 0 \right\rrangle} \overline{p_0}^c \overline{p_1} p_1 \tag*{\textbf{[R.1]}} \\
    &=~ \overline{\left\llangle \overline{p_1}^c p_0, 0 \right\rrangle} \overline{p_1} \overline{p_0}^c  p_1 \tag*{Lemma \ref{lem:rest}.(\ref{lem:rest.idem.R2})} \\
    &=~ \overline{\left\llangle \overline{p_1}^c p_0, 0 \right\rrangle p_1}   \overline{p_0}^c  p_1 \tag*{\textbf{[R.3]}} \\
&=~ \overline{0}   \overline{p_0}^c  p_1 \tag*{(\ref{diag:pair2})} \\
&=~ 0 \tag*{Rest. zero}
\end{align*}
So by universal property of the product, it follows that $\overline{\left\llangle \overline{p_1}^c p_0, 0 \right\rrangle} \left\llangle 0, \overline{p_0}^c p_1 \right\rrangle = \llangle 0,0 \rrangle$. By (\ref{lem:class-prod.0}), this means that $\overline{\left\llangle \overline{p_1}^c p_0, 0 \right\rrangle} \left\llangle 0, \overline{p_0}^c p_1 \right\rrangle = 0$. Thus $\left\llangle \overline{p_1}^c p_0, 0 \right\rrangle \perp \left\llangle 0, \overline{p_0}^c p_1 \right\rrangle$ as desired. So we can take their join and compute: 
\begin{align*}
    \left( \left\llangle \overline{p_1}^c p_0, 0 \right\rrangle \vee \left\llangle 0, \overline{p_0}^c p_1 \right\rrangle \right) p_0 &=~ \left\llangle \overline{p_1}^c p_0, 0 \right\rrangle p_0 \vee \left\llangle 0, \overline{p_0}^c p_1 \right\rrangle p_0 \tag*{Lemma \ref{lem:join}.(\ref{lem:join.iii})}  \\
    &=~ \overline{p_1}^c p_0 p_0 \vee 0  \tag*{(\ref{diag:pair2})} \\
    &=~ \overline{p_1}^c p_0  \tag*{Lemma \ref{lem:join}.(\ref{lem:join.ii})} 
\end{align*}
So we have the equality $\left( \left\llangle \overline{p_1}^c p_0, 0 \right\rrangle \vee \left\llangle 0, \overline{p_0}^c p_1 \right\rrangle \right) p_0= \overline{p_1}^c p_0$ holds, and similarly we can also compute that $\left( \left\llangle \overline{p_1}^c p_0, 0 \right\rrangle \vee \left\llangle 0, \overline{p_0}^c p_1 \right\rrangle \right) p_1= \overline{p_0}^c p_1$. Thus by the universal property of the product, it follows that $\left\llangle \overline{p_1}^c p_0, 0 \right\rrangle \vee \left\llangle 0, \overline{p_0}^c p_1 \right\rrangle = \llangle \overline{p_1}^c p_0, \overline{p_0}^c p_1 \rrangle$. By (\ref{lem:class-prod.vi}), it follows that $\left\llangle \overline{p_1}^c p_0, 0 \right\rrangle  \vee  \left\llangle 0, \overline{p_0}^c p_1 \right\rrangle = \overline{p_0}^c \vee \overline{p_1}^c$ as desired. 
\item These follow from (\ref{lem:class-prod.iii}) and Lemma \ref{lem:ecomp}.(\ref{lem:ecomp.viii}). 
\item First observe that $\overline{p_0}^c p_1 \llangle 0,1_B \rrangle p_0 = 0$ and $\overline{p_0}^c p_1 \llangle 0,1_B \rrangle p_1 = \overline{p_0}^c p_1$. By Lemma \ref{lem:ecomp}.(\ref{lem:ecomp.i}), we also that $\overline{p_0}^c p_0 =0$. Thus by the universal property of the product, it follows that $\overline{p_0}^c p_1 \llangle 0,1_B \rrangle = \overline{p_0}^c$. By Lemma \ref{lem:class-prod}.(\ref{lem:class-prod.v}), we also have that $ \llangle 0,1_B \rrangle \overline{p_0}^c p_1 = 1_B$. Therefore, we have that $\overline{p_0}^c$ splits, and similarly we can also show that $\overline{p_1}^c$ splits. 
\end{enumerate}
\end{proof}

\begin{lemma} \label{lem:prod-restcoprod} Let $\mathbb{X}$ be a classical restriction category with binary products. Then $\mathbb{X}$ has finite restriction coproducts if and only if for every pair of objects $A$ and $B$, the restriction idempotent $\overline{p_0}^c \vee \overline{p_1}^c: A \& B \to A \& B$ splits. Explicitly: 
\begin{enumerate}[{\em (i)}]
\item If the restriction coproduct $A \oplus B$ exists, then the restriction idempotent $\overline{p_0}^c \vee \overline{p_1}^c$ splits via the maps $ \overline{p_1}^c p_0 \iota_0 \vee \overline{p_0}^c p_1 \iota_1: A \& B \to A \oplus B$ and $\overline{p_1}^c p_0 \iota_0 \vee \overline{p_0}^c p_1 \iota_1: A \& B \to A \oplus B$. 
\item If the restriction idempotent $\overline{p_0}^c \vee \overline{p_1}^c$ splits via maps $r: A \& B \to A \oplus B$ and ${s: A \oplus B \to A \& B}$, then $A \oplus B$ is a restriction product of $A$ and $B$ with injections $\iota_0: A \to A\oplus B$ and $\iota_1: B \to A \oplus B$ defined as $\iota_0 = \llangle 1_A, 0 \rrangle r$ and $\iota_1 = \llangle 0, 1_B \rrangle r$.  
\end{enumerate}
\end{lemma}
\begin{proof} 
\begin{enumerate}[{\em (i)}]
\item Suppose that $A \oplus B$ is the restriction coproduct of $A$ and $B$. We first check that $\overline{p_1}^c p_0 \iota_0$ and $\overline{p_0}^c p_1 \iota_1$ are indeed disjoint:
\begin{align*}
    \overline{\overline{p_1}^c p_0 \iota_0} \overline{p_0}^c p_1 \iota_1 &=~ \overline{\overline{p_1}^c p_0} \overline{p_0}^c p_1 \iota_1 \tag*{Lemma \ref{lem:rest}.(\ref{lem:rest.total.2}) and $\iota_0$ total}  \\
    &=~ \overline{p_1}^c  \overline{p_0}  \overline{p_0}^c p_1 \iota_1 \tag*{Lemma \ref{lem:rest}.(\ref{lem:rest.idem.R3}) and Lemma \ref{lem:ecomp}.(\ref{lem:ecomp.0})}  \\
    &=~ 0 \tag*{Lemma \ref{lem:ecomp}.(\ref{lem:ecomp.ii})}
\end{align*}
So $\overline{p_1}^c p_0 \iota_0 \perp \overline{p_0}^c p_1 \iota_1$, thus we can indeed take their join $\overline{p_1}^c p_0 \iota_0 \vee \overline{p_0}^c p_1 \iota_1$. Now we must show both equalities that $\left( \overline{p_1}^c p_0 \iota_0 \vee \overline{p_0}^c p_1 \iota_1 \right) \llangle \iota^\circ_0, \iota^\circ_1 \rrangle = \overline{p_0}^c \vee \overline{p_1}^c$ and $\llangle \iota^\circ_0, \iota^\circ_1 \rrangle \left( \overline{p_1}^c p_0 \iota_0 \vee \overline{p_0}^c p_1 \iota_1 \right)  = 1_{A \oplus B}$. So we compute: 
\begin{align*}
   \left( \overline{p_1}^c p_0 \iota_0 \vee \overline{p_0}^c p_1 \iota_1 \right) \llangle \iota^\circ_0, \iota^\circ_1 \rrangle &=~ \overline{p_1}^c p_0 \iota_0\llangle \iota^\circ_0, \iota^\circ_1 \rrangle  \vee \overline{p_0}^c p_1 \iota_1 \llangle \iota^\circ_0, \iota^\circ_1 \rrangle  \tag*{Lemma \ref{lem:join}.(\ref{lem:join.iii})} \\
   &=~   \left\llangle \overline{p_1}^c p_0\iota_0\iota^\circ_0, \overline{p_1}^c p_0\iota_0\iota^\circ_1 \right\rrangle  \vee  \left\llangle \overline{p_0}^c p_1\iota_1 \iota^\circ_0, \overline{p_0}^c p_1\iota_1 \iota^\circ_1 \right\rrangle \tag*{(\ref{eq:pair-comp})} \\
   &=~ \left\llangle \overline{p_1}^c p_0, 0 \right\rrangle  \vee  \left\llangle 0, \overline{p_0}^c p_1 \right\rrangle \tag*{Lemma \ref{lem:restcoprodzero}.(\ref{lem:restcoprodzero.iii})} \\ 
   &=~  \overline{p_0}^c \vee \overline{p_1}^c \tag*{Lemma \ref{lem:class-prod}.(\ref{lem:class-prod.ii})} 
\end{align*}
\begin{align*}
  \llangle \iota^\circ_0, \iota^\circ_1 \rrangle \left( \overline{p_1}^c p_0 \iota_0 \vee \overline{p_0}^c p_1 \iota_1 \right) &=~  \llangle \iota^\circ_0, \iota^\circ_1 \rrangle\overline{p_1}^c p_0 \iota_0 \vee  \llangle \iota^\circ_0, \iota^\circ_1 \rrangle\overline{p_0}^c p_1 \iota_1 \tag*{\textbf{[J.3]}} \\
&=~ \overline{\iota^\circ_1}^c \iota^\circ_0 \iota_0 \vee \overline{\iota^\circ_0}^c \iota^\circ_1 \iota_1 \tag*{Lemma \ref{lem:class-prod}.(\ref{lem:class-prod.iv})} \\
&=~ \overline{\iota^\circ_1}^c  \vee  \overline{\iota^\circ_0}^c \tag*{Lemma \ref{lem:restcoprodzero}.(\ref{lem:restcoprodzero.iii})} \\ 
  &=~ \left(\overline{\iota^\circ_1} ~ \overline{\iota^\circ_0} \right)^c  \tag*{Lemma \ref{lem:ecomp}.(\ref{lem:ecomp.vii})} \\
  &=~ 0^c \tag*{Lemma \ref{lem:restcoprodzero}.(\ref{lem:restcoprodzero.v})} \\ 
  &=~ 1_{A \oplus B} \tag*{Lemma \ref{lem:ecomp}.(\ref{lem:ecomp.viii})} \\
\end{align*}
So we conclude that $\overline{p_0}^c \vee \overline{p_1}^c$ splits. 
\item Suppose that $\overline{p_0}^c \vee \overline{p_1}^c$ splits via $r: A \& B \to A \oplus B$ and $s: A \oplus B \to A \& B$, so: 
\begin{align}\label{ppcsplit}
    rs = \overline{p_0}^c \vee \overline{p_1}^c && sr = 1_{A \oplus B}
\end{align}
We must first show that the suggested injections are total. First note again that since $s$ is a section, $s$ is monic, and thus $s$ is total (Lemma \ref{lem:rest}.(\ref{lem:rest.total})). Therefore it follows from Lemma \ref{lem:rest}.(\ref{lem:rest.total.2}) that:
\begin{align}\label{rrest}
    \overline{r} = \overline{p_0}^c \vee \overline{p_1}^c 
\end{align}
We then compute that: 
\begin{align*}
    \overline{\iota_0} &=~ \overline{\llangle 1_A, 0 \rrangle r} \tag*{Def. of $\iota_0$} \\
    &=~ \overline{\llangle 1_A, 0 \rrangle \overline{r}} \tag*{Lemma \ref{lem:rest}.(\ref{lem:rest.i})} \\ 
    &=~ \overline{\llangle 1_A, 0 \rrangle \left(\overline{p_0}^c \vee \overline{p_1}^c \right)} \tag*{(\ref{rrest})} \\ 
    &=~ \overline{\llangle 1_A, 0 \rrangle \overline{p_0}^c \vee \llangle 1_A, 0 \rrangle \overline{p_1}^c} \tag*{\textbf{[J.3]}} \\
    &=~ \overline{0 \vee \llangle 1_A, 0 \rrangle} \tag*{Lemma \ref{lem:class-prod}.(\ref{lem:class-prod.v})}  \\
    &=~ \overline{\llangle 1_A, 0 \rrangle} \tag*{Lemma \ref{lem:join}.(\ref{lem:join.ii})}  \\
    &=~ \overline{\llangle 1_A, \overline{1_A} 0 \rrangle} \tag*{$1_A$ total} \\
  &=~ \overline{\llangle 1_A, 0 \rrangle \overline{p_0}}  \tag*{Lemma \ref{lem:rest-product}.(\ref{lem:rest-product.ii})} \\
    &=~ \overline{\llangle 1_A, 0 \rrangle p_0} \tag*{Lemma \ref{lem:rest}.(\ref{lem:rest.i})} \\
    &=~ \overline{1_A} \tag*{(\ref{diag:pair2})} \\
    &=~ 1_A \tag*{$1_A$ total}
\end{align*} 
So $\overline{\iota_0}= 1_A$, and similarly we can show that $\overline{\iota_1}= 1_B$. So the injection maps are indeed total. Now for a pair of maps $f: A \to C$ and $g: B \to C$, consider the composites $\overline{p_1}^c p_0 f$ and $\overline{p_0}^c p_1 g$. We show these are disjoint: 
\begin{align*}
   \overline{ \overline{p_1}^c p_0 f } \overline{p_0}^c p_1 g &=~ \overline{p_0 f}  \overline{p_1}^c \overline{p_0}^c p_1 g \tag*{Lemma \ref{lem:rest}.(\ref{lem:rest.idem.R3}) and Lemma \ref{lem:ecomp}.(\ref{lem:ecomp.0})}  \\
   &=~ 0 \tag*{Lemma \ref{lem:class-prod}.(\ref{lem:class-prod.i})} 
\end{align*}
So $\overline{p_1}^c p_0 f \perp \overline{p_0}^c p_1 g$. Thus we can take their join and define $[f,g]: A \oplus B \to C$ as follows: 
\begin{align}
    [f,g] := s \left( \overline{p_1}^c p_0 f \vee \overline{p_0}^c p_1 g \right)
\end{align}
Then we compute: 
\begin{align*}
    \iota_0 [f,g] &=~ \llangle 1_A, 0 \rrangle r  s \left( \overline{p_1}^c p_0 f \vee \overline{p_0}^c p_1 g \right) \tag*{Def. of $\iota_0$ and $[f,g]$}  \\
    &=~  \left(\overline{p_0}^c \vee \overline{p_1}^c \right) \left( \overline{p_1}^c p_0 f \vee \overline{p_0}^c p_1 g \right) \tag*{(\ref{ppcsplit})} \\
    &=~ \left(\llangle 1_A, 0 \rrangle\overline{p_0}^c \vee \llangle 1_A, 0 \rrangle\overline{p_1}^c \right) \left( \overline{p_1}^c p_0 f \vee \overline{p_0}^c p_1 g \right) \tag*{\textbf{[J.3]}} \\ 
    &=~ \left(0 \vee \llangle 1_A, 0 \rrangle \right) \left( \overline{p_1}^c p_0 f \vee \overline{p_0}^c p_1 g \right) \tag*{Lemma \ref{lem:class-prod}.(\ref{lem:class-prod.v})}  \\
      &=~ \llangle 1_A, 0 \rrangle  \left( \overline{p_1}^c p_0 f \vee \overline{p_0}^c p_1 g \right) \tag*{Lemma \ref{lem:join}.(\ref{lem:join.ii})}  \\
      &=~ \llangle 1_A, 0 \rrangle \overline{p_1}^c p_0 f \vee \llangle 1_A, 0 \rrangle \overline{p_0}^c p_1 g \tag*{\textbf{[J.3]}} \\ 
&=~   \llangle 1_A, 0 \rrangle p_0 f \vee 0 \tag*{Lemma \ref{lem:class-prod}.(\ref{lem:class-prod.v})}  \\
&=~  \llangle 1_A, 0 \rrangle p_0 f \tag*{Lemma \ref{lem:join}.(\ref{lem:join.ii})}  \\
&=~ f \tag*{(\ref{diag:pair2})}
\end{align*}
So $\iota_0 [f,g]=f$, and similarly we can show that $\iota_1 [f,g]=g$. Lastly for uniqueness, suppose that there is a map $h: A \oplus B \to C$ such that $\iota_0 h =f$ and $\iota_1 h =g$. Now we first compute that: 
\begin{align*}
    s p_0 &=~ srs p_0 \tag*{(\ref{ppcsplit})} \\
    &=~ s\left(\overline{p_0}^c \vee \overline{p_1}^c \right) p_0 \\
    &=~ s\left(\overline{p_0}^c p_0 \vee \overline{p_1}^c p_0 \right) \tag*{Lemma \ref{lem:join}.(\ref{lem:join.iii})}  \\
    &=~ s\left(0 \vee \overline{p_1}^c p_0 \right) \tag*{Lemma \ref{lem:ecomp}.(\ref{lem:ecomp.i})}  \\
    &=~ s \overline{p_1}^c p_0 \tag*{Lemma \ref{lem:join}.(\ref{lem:join.ii})} 
\end{align*}
So $s p_0=s \overline{p_1}^c p_0$, and similarly we can show that $s p_1= s \overline{p_0}^c p_1$. Then we compute: 
\begin{align*}
h &=~ s r s r  h \tag*{(\ref{ppcsplit})} \\
&=~ s \left( \overline{p_0}^c \vee \overline{p_1}^c \right) r h \tag*{(\ref{ppcsplit})} \\
&=~ s \left( \llangle \overline{p_1}^c p_0 , 0 \rrangle   \vee  \llangle 0,\overline{p_0}^c p_1 \rrangle   \right) r h \tag*{Lemma \ref{lem:class-prod}.(\ref{lem:class-prod.ii})} \\
&=~ s \left( \llangle \overline{p_1}^c p_0 , 0 \rrangle rh  \vee  \llangle 0,\overline{p_0}^c p_1 \rrangle rh  \right) \tag*{Lemma \ref{lem:join}.(\ref{lem:join.iii})}  \\
&=~ s \left( \overline{p_1}^c p_0 \llangle 1_A , 0 \rrangle rh  \vee  \overline{p_0}^c p_1 \llangle 0, 1_B \rrangle rh  \right) \tag*{(\ref{eq:pair-comp})} \\
&=~ s \left( \overline{p_1}^c p_0 \iota_0 h  \vee  \overline{p_0}^c p_1 \iota_1 h  \right) \tag*{Def. of $\iota_j$} \\
&=~ s \left( \overline{p_1}^c p_0 f \vee \overline{p_0}^c p_1 g \right) \tag*{Assump. on $h$} \\
&=~  [f,g]\tag*{Def. of $[f,g]$}
\end{align*}
So $[f,g]$ is unique. Thus we conclude that $A \oplus B$ is a restriction coproduct. 
\end{enumerate}
\end{proof}

We will now explain how in a classical restriction category with both products $\&$ and restriction products $\times$, it turns out that the product $A \& B$ is always the restriction coproduct of $A$, $B$, and $A \times B$, that is, $A \& B = A \oplus B \oplus (A \times B)$. To do so, we will make use of the following: 

\begin{lemma} \label{lem:coprodsplit} \cite[Lemma 7.7]{cockett2009boolean} Let $\mathbb{X}$ be a classical restriction category, and let ${e_0: A \to A}$, $e_1: A \to A$, and $e_2: A \to A$ be split restriction idempotents, with splittings $r_i: A \to A_i$ and $s_i: A_i \to A$, that are pairwise disjoint and such that $e_0 \vee e_1 \vee e_2 = 1_A$. Then $A$ is a restriction coproduct of $A_0$, $A_1$, and $A_2$ where the injections are $s_i: A_i \to A$, and furthermore, for maps $f_0: A_0 \to C$, $f_1: A_1 \to C$, and $f_2: A_2 \to C$, their copairing $[f_0,f_1,f_2]: A \to C$ is defined as $[f_0, f_1, f_2] = r_0 f_0 \vee r_1 f_1 \vee r_2 f_2$. 
\end{lemma}

\begin{proposition} Let $\mathbb{X}$ be a classical restriction category with binary products and binary restriction products. Then for each object $A$ and $B$, their product $A \& B$ is a restriction coproduct of $A$, $B$, and $A \times B$ with injection maps ${\iota_0: A \to A\&B}$, $\iota_1: B \to A\&B$, and $\iota_2: A \times B \to A\& B$ respectively defined as follows: 
\begin{align*}
    \iota_0 = \llangle 1_A, 0 \rrangle && \iota_1 = \llangle 0,1_B \rrangle && \iota_2 = \llangle \pi_0, \pi_1 \rrangle 
\end{align*}
and furthermore, for maps $f: A \to C$, $g: B \to C$, and $h: A \times B \to C$, their copairing $[f,g,h]: A \& B \to C$ is defined as follows: 
\[ [f,g,h] = \overline{p_1}^c p_0 f \vee \overline{p_0}^c p_1 g \vee \langle p_0, p_1 \rangle h  \]
Moreover, the following equalities hold: 
\begin{align*}
    p_0 = [1_A, 0, \pi_0] && p_1 = [0, 1_B, \pi_1] 
\end{align*}
\end{proposition}
\begin{proof} We wish to make use of Lemma \ref{lem:coprodsplit}. To do so, we need to provide three split restriction idempotents of type $A \& B$, that are pairwise disjoint, and whose join is equal to the identity. So consider the restriction idempotents $e_0 := \overline{p_1}^c$, $e_1 := \overline{p_0}^c$, and $e_2 := \overline{p_0}~ \overline{p_1}$. By Lemma \ref{lem:class-prod}.(\ref{lem:class-prod.vii}), $\overline{p_1}^c$ splits through $A$ via $r_0 := \overline{p_1}^c p_0$ and $s_0 := \llangle 1_A, 0 \rrangle$, while $\overline{p_0}^c$ splits through $B$ via $r_1 := \overline{p_0}^c p_1$ and $s_1 := \llangle 0, 1_B \rrangle$. By Lemma \ref{lem:prod-restprod}, we also know that $\overline{p_0}~ \overline{p_1}$ splits through $A \times B$ via $r_2 := \langle p_0, p_1 \rangle$ and $s_2 := \left\llangle \pi_0, \pi_1 \right\rrangle$. Now by Lemma \ref{lem:class-prod}.(\ref{lem:class-prod.i}), we have that $\overline{p_0}^c \perp \overline{p_1}^c$, and by Lemma \ref{lem:ecomp}.(\ref{lem:ecomp.ii}), it follows that $\overline{p_0}~ \overline{p_1} \perp \overline{p_j}^c$ as well. So our chosen restriction idempotents are pairwise disjoint. Taking their join, we compute that: 
\begin{align*}
   \overline{p_1}^c \vee \overline{p_0}^c \vee \overline{p_0}~ \overline{p_1} &=~  \overline{p_0}~ \overline{p_1} \vee  \overline{p_0}^c \vee \overline{p_1}^c \\
    &=~   \overline{p_0}~ \overline{p_1} \vee (  \overline{p_0}~ \overline{p_1})^c \tag*{Lemma \ref{lem:ecomp}.(\ref{lem:ecomp.ii})} \\
    &=~ 1_{A \& B} \tag*{Lemma \ref{lem:ecomp}.(\ref{lem:ecomp.ii})}
\end{align*}
    Thus $\overline{p_1}^c \vee \overline{p_0}^c \vee \overline{p_0}~ \overline{p_1} =  1_{A \& B}$ as desired. Thus, we can apply Lemma \ref{lem:coprodsplit} to obtain that $A \& B$ is the restriction coproduct of $A$, $B$, and $A \times B$, with the injections given by $\iota_0 = \llangle 1_A, 0 \rrangle$, $\iota_1 = \llangle 0,1_B \rrangle$, and $\iota_2 = \llangle \pi_0, \pi_1 \rrangle$, and the copairing by $[f,g,h] = \overline{p_1}^c p_0 f \vee \overline{p_0}^c p_1 g \vee \langle p_0, p_1 \rangle h$, as desired. Now by definition, we also have that $\iota_0 p_0 = 1_A$, $\iota_1 p_0 = 0$, and $\iota_2 p_0 = \pi_0$. Thus by the couniversal of the coproduct, it follows that $p_0 = [1_A, 0, \pi_0]$. Similarly, we can also show that $p_1 = [0, 1_B, \pi_1]$.
\end{proof}

Thus if a distributive restriction category is classical and has finite products, said product and projections must be the classical product and classical projections.

\section{Classical Distributive Restriction Categories}\label{sec:CDRC}

In this section we prove the main result of this paper that a distributive restriction category is classical if and only if it has classical products. 

We begin by showing that a classical distributive category (by which we simply mean a distributive restriction category that is also classical, and no further assumptions) has classical products. First, here is a useful lemma about the restrictions of the projections and their complements. 

\begin{lemma}\label{lem:cdrc} In a classic  distributive restriction category $\mathbb{X}$, 
    \begin{enumerate}[{\em (i)}]
\item \label{lem:cdrc.0} $\overline{p_0} = \iota^\circ_0\iota_0 \vee \iota^\circ_2\iota_2$ and $\overline{p_0} = \iota^\circ_1\iota_1 \vee \iota^\circ_2\iota_2$; 
\item \label{lem:cdrc.i} $\overline{p_0}^c = \iota^\circ_1\iota_1$ and $\overline{p_1}^c = \iota^\circ_0\iota_0$;
\item \label{lem:cdrc.ii} $\overline{p_1}^c p_0 = \iota^\circ_0$ and $\overline{p_0}^c p_1 = \iota^\circ_1$
\end{enumerate}
\end{lemma}
\begin{proof}
    \begin{enumerate}[{\em (i)}]
    \item  This follows from Lemma \ref{lem:cprod}.(\ref{lem:cprod.i}) and Lemma \ref{lem:class-coprod}.(\ref{lem:class-coprod.ii}), so $\overline{p_0} = 1_A \oplus  0 \oplus  1_{A \times B} = \iota^\circ_0\iota_0 \vee \iota^\circ_2\iota_2$ and $\overline{p_1} = 0 \oplus 1_B \oplus  1_{A \times B} = \iota^\circ_1\iota_1 \vee \iota^\circ_2\iota_2$. 
\item To show this, we make use of Lemma \ref{lem:ecomp}.(\ref{lem:ecomp.iv}). Indeed, recall that by Lemma \ref{lem:restcoprodzero}.(\ref{lem:restcoprodzero.iv}), $\iota^\circ_1\iota_1 = \overline{\iota^\circ_1}$ is a restriction idempotent. So we compute: 
\begin{align*}
   \overline{p_0} \iota^\circ_1\iota_1 &=~ \left(1_A \oplus  0 \oplus  1_{A \times B} \right)\left(0 \oplus 1_B \oplus 0 \right) \tag*{Lemma \ref{lem:cprod}.(\ref{lem:cprod.i}) and Lemma \ref{lem:restcoprodzero}.(\ref{lem:restcoprodzero.ii})} \\
   &=~ 0 \oplus 0 \oplus 0 \tag*{(\ref{eq:copair-comp})} \\
   &=~ 0 
\end{align*}
\begin{align*}
    \overline{p_0} \vee  \iota^\circ_1\iota_1 &=~ \iota^\circ_0\iota_0 \vee \iota^\circ_2\iota_2 \vee  \iota^\circ_1\iota_1 \tag*{Lemma \ref{lem:cdrc}.(\ref{lem:cdrc.0})} \\
    &=~ \iota^\circ_0\iota_0  \vee  \iota^\circ_1\iota_1 \vee \iota^\circ_2\iota_2 \\
    &=~ 1_{A \oplus B \oplus (A \times B)} \tag*{Lemma \ref{lem:class-coprod}.(\ref{lem:class-coprod.iv})}
\end{align*}
Then by Lemma \ref{lem:ecomp}.(\ref{lem:ecomp.iv}), it follows that $\overline{p_0}^c = \iota^\circ_1\iota_1$. Similarly, we can show that $\overline{p_1}^c = \iota^\circ_0\iota_0$.
\item By definition of the classical projections and (\ref{lem:cdrc.i}), it follows that $\overline{p_1}^c p_0 = \iota^\circ_0$ and $\overline{p_0}^c p_1 = \iota^\circ_1$ as desired. 
\end{enumerate}
\end{proof}

\begin{proposition}\label{prop:classtoclassprod} A classical distributive restriction category has classical products.
\end{proposition}
\begin{proof} Let $\mathbb{X}$ be a classical distributive restriction category. Since $\mathbb{X}$ has restriction zeroes by assumption, we need only show that $A \oplus  B \oplus  (A \times B)$ is a product of $A$ and $B$. So let $f: C \to A$ and ${g: C \to B}$ be a given pair of maps. First consider the parallel maps $\overline{g}^c f\iota_0$, $\overline{f}^c g\iota_1$, and $\langle f,g \rangle \iota_2$. We first check that: 
\begin{align*}
\overline{\overline{g}^c f\iota_0} \overline{f}^c g\iota_1 &=~ \overline{\overline{f}^c\overline{g}^c f\iota_0} g\iota_1 \tag*{Lemma \ref{lem:rest}.(\ref{lem:rest.idem.R3})} \\
&=~ \overline{\overline{g}^c \overline{f}^c f\iota_0} g\iota_1 \tag*{Lemma \ref{lem:rest}.(\ref{lem:rest.idem.R2})}  \\ 
&=~ \overline{\overline{g}^c  0 \iota_0} g\iota_1 \tag*{Lemma \ref{lem:ecomp}.(\ref{lem:ecomp.i})} \\
&=~ 0 \tag*{Rest. zero map}
\end{align*}
So $\overline{g}^c f\iota_0 \perp \overline{f}^c g\iota_1$. To show that $\overline{g}^c f\iota_0$ and $\langle f,g \rangle \iota_2$ are disjoint, recall that the restriction of a pair is given by the meet of the restrictions \cite[Prop 2.8]{cockett2012differential}: 
\begin{align}\label{eq:rest-pair}
    \overline{\langle f,g \rangle} = \overline{f}~\overline{g}
\end{align}
So we can compute: 
\begin{align*}
    \overline{\langle f,g \rangle \iota_2} \overline{g}^c f\iota_0 &=~ \overline{\langle f,g \rangle} \overline{g}^c f\iota_0 \tag*{Lemma \ref{lem:rest}.(\ref{lem:rest.total.2}) and $\iota_2$ total}  \\
    &=~ \overline{f} \overline{g} ~\overline{g}^c f\iota_0 \tag*{(\ref{eq:rest-pair})}  \\
    &=~ 0 \tag*{Lemma \ref{lem:ecomp}.(\ref{lem:ecomp.ii})} 
\end{align*}
So $ \overline{g}^c f\iota_0 \perp \langle f,g \rangle \iota_2$, and similarly we can show that $\overline{f}^c g\iota_1 \perp \langle f,g \rangle \iota_2$. So $\overline{g}^c f\iota_0$, $\overline{f}^c g\iota_1$, and $\langle f,g \rangle \iota_2$ are pair-wise disjoint (and thus compatible). Then define the classical pairing $\llangle f, g \rrangle: C \to A \oplus  B \oplus  (A \times B)$ as the join of those three maps: 
\begin{align} \llangle f, g \rrangle := \overline{g}^c f\iota_0 \vee \overline{f}^c g\iota_1 \vee \langle f,g \rangle \iota_2
\end{align}
We now compute that: 
\begin{align*}
\llangle f, g \rrangle p_0 &=~ \left( \overline{g}^c f\iota_0 \vee \overline{f}^c g\iota_1 \vee \langle f,g \rangle \iota_2 \right) p_0 \tag*{(\ref{diag:pair2})}\\
&=~  \overline{g}^c f\iota_0p_0 \vee \overline{f}^c g\iota_1p_0 \vee \langle f,g \rangle \iota_2p_0 \tag*{Lemma \ref{lem:join}.(\ref{lem:join.iii})} \\
&=~ \overline{g}^c f \vee 0 \vee \langle f,g \rangle\pi_0 \tag*{Def. of $p_0$} \\
&=~ \overline{g}^c f \vee \overline{g}f \tag*{Lemma \ref{lem:join}.(\ref{lem:join.ii})} \\
&=~ (\overline{g}^c \vee \overline{g})f \tag*{Lemma \ref{lem:join}.(\ref{lem:join.iii})} \\
&=~ f \tag*{Lemma \ref{lem:ecomp}.(\ref{lem:ecomp.ii})} \\
\end{align*}
So $\llangle f, g \rrangle p_0 = f$, and similarly we can show that $\llangle f, g \rrangle p_1 = g$. Now we must show the uniqueness of $\llangle f,g \rrangle$. So let $h: C \to A \oplus  B \oplus  (A \times B)$ be a map such that $h p_0 = f$ and $h p_1 =g$. Then we first compute that:  
\begin{align*} 
h \iota^\circ_0 &=~ h \overline{p_1}^c p_0 \tag*{Lemma \ref{lem:cdrc}.(\ref{lem:cdrc.ii})} \\
&=~ \overline{h p_1}^c h p_0  \tag*{Lemma \ref{lem:ecomp}.(\ref{lem:ecomp.vi})} \\
&=~ \overline{g}^cf \tag*{Assump. on $h$}
\end{align*}
So $h \iota^\circ_0= \overline{g}^cf$ and similarly we can show that $h \iota^\circ_1 = \overline{f}^c g$. We also compute that: 
\begin{align*}
 h \iota^\circ_2 \pi_0 &=~ h \overline{p_1} p_0  \tag*{Lemma \ref{lem:cprod}.(\ref{lem:cprod.ii})} \\
 &=~ \overline{h p_1} h p_0 \tag*{\textbf{[R.4]}} \\
 &=~ \overline{g} f \tag*{Assump. on $h$}
\end{align*}
So $h \iota^\circ_2 \pi_0 = \overline{g} f$ and similarly $h \iota^\circ_2 \pi_1 = \overline{f} g$. Then by universal property of the restriction product, we have that $h \iota^\circ_2 = \langle f,g \rangle$. So finally we compute that: 
\begin{align*} 
h &=~ h\left( \iota^\circ_0 \iota_0 \vee \iota^\circ_1\iota_1 \vee \iota^\circ_2 \iota_2 \right) \tag*{Lemma \ref{lem:class-coprod}.(\ref{lem:class-coprod.iv})} \\
&=~ h\iota^\circ_0 \iota_0 \vee h\iota^\circ_1\iota_1 \vee h\iota^\circ_2 \iota_2 \tag*{Lemma \ref{lem:join}.(\ref{lem:join.iii})} \\
&=~  \overline{g}^c f\iota_0 \vee \overline{f}^c g\iota_1 \vee \langle f,g \rangle \iota_2 \tag*{Identities for $h$} \\
&=~ \llangle f, g \rrangle
\end{align*}
Thus $A \oplus  B \oplus  (A \times B)$ is a product of $A$ and $B$ as desired. So we conclude that $\mathbb{X}$ has classical products. \end{proof}

We now wish to prove the converse, that if a distributive restriction category has classical products then it is classical. To do so, we will need to make use of \emph{decisions}. 

\begin{definition}\label{def:decision} An \textbf{extensive restriction category} \cite[Sec 3]{cockett2007restriction} is a coCartesian restriction category $\mathbb{X}$ with restriction zeroes such that for every for every map $f: A \to B_0 \oplus  \hdots \oplus  B_n$ there exists a (necessarily unique) map ${\mathsf{d}[f]: A \to A \oplus  \hdots \oplus  A}$, called the \textbf{decision} \cite[Prop 2.11]{cockett2007restriction} of $f$, such that the following equalities hold: 
\begin{enumerate}[{\bf [D.1]}]
\item $\overline{f} = \mathsf{d}[f] \left[1_A, \hdots, 1_A \right]$ 
\item $\mathsf{d}[f] (f \oplus  \hdots \oplus  f) = f (\iota_0 \oplus  \hdots \oplus  \iota_n)$
\end{enumerate}
\end{definition}

Every distributive restriction category with restriction zeroes is extensive. While for the proof of Prop \ref{prop:cptoc}, just knowing that decisions exist is sufficient, it may be useful to record explicitly how decisions are constructed and some examples. 

\begin{proposition} \cite[Thm 5.8]{cockett2007restriction} A distributive restriction category with restriction zeroes is extensive, where for a map of type $f: A \to B_0 \oplus  \hdots \oplus  B_n$, its decision ${\mathsf{d}[f]: A \to A \oplus  \hdots \oplus  A}$ is defined as the following composite\footnote{This fixes the minor typo in the journal version \url{http://www.tac.mta.ca/tac/volumes/42/6/42-06abs.html}, where we copied down the incorrect formula for (\ref{def:dec}) from another reference. Fortunately, as mentioned above, we do not use this formula in anywhere and it was only for exposition. So the rest of the paper remains unchanged.}: 
 \begin{equation}\begin{gathered}\label{def:dec}
\begin{array}[c]{c} \mathsf{d}[f] \end{array}
\!\!\!:=\! \!\!\! \begin{array}[c]{c} \xymatrixcolsep{4.75pc}\xymatrixrowsep{1pc}\xymatrix{ A \ar[r]^-{ \left\langle 1_A, f \right\rangle }  & A \times \left( B_0 \oplus  \hdots \oplus  B_n \right) \ar[r]^-{ 1_A \times \left(t_{B_0} \oplus \hdots \oplus t_{B_n}\right)} & A \times (\mathsf{1} \oplus \hdots \oplus \mathsf{1}) \cong A \oplus \hdots \oplus A } \end{array}  \end{gathered}\end{equation}
\end{proposition}

\begin{example} \normalfont  In $\mathsf{PAR}$, the decision of $f: X \to Y_0 \sqcup \hdots \sqcup Y_n$ is the partial function ${\mathsf{d}[f]: X \to X \sqcup \hdots \sqcup X}$ defined as follows: 
\[ \mathsf{d}[f] (x) = \begin{cases} \iota_0(x) & \text{if } f(x) \downarrow \text{ and } f(x) \in Y_0 \\
\vdots \\
\iota_n(x) & \text{if } f(x) \downarrow \text{ and } f(x) \in Y_n \\
\uparrow & \text{o.w. }  \end{cases} \]
\end{example}

\begin{example} \normalfont  In $k\text{-}\mathsf{CALG}_\bullet$, the codecision of $f: A_0 \times \hdots \times A_n \to B$ is the non-unital $k$-algbera morphism $\mathsf{d}[f]: B \times \hdots \times B \to B$ defined as follows: 
\[ \mathsf{d}[f](b_0, \hdots, b_n) = \sum^{n}_{k=0} f(0,\hdots, 0,\underset{\substack{k\text{-th}\\\text{term}}}{1},0, \hdots,0) b \]
\end{example}

\begin{proposition}\label{prop:cptoc} A distributive restriction category with classical products is classical. 
\end{proposition}
\begin{proof}  Let $\mathbb{X}$ be a distributive restriction category with classical products. To prove that $\mathbb{X}$ is classical, we need to show that $\mathbb{X}$ has joins and relative complements. Starting with the joins, since $\mathbb{X}$ has restriction zeroes by assumption, to prove that $\mathbb{X}$ has finite joins, it suffices to prove that $\mathbb{X}$ has binary joins \cite[Lemma 6.8]{cockett2009boolean}. So let $f: A \to B$ and $g: A \to B$ be compatible maps, $f \smile g$. Define $f \vee g: A \to B$ as follows: 
 \begin{equation}\begin{gathered}\label{def:fgjoin1}
f \vee g := \xymatrixcolsep{5pc}\xymatrix{ A \ar[r]^-{ \llangle f, g \rrangle }  & B \oplus  B \oplus  (B \times B) \ar[r]^-{[ 1_B, 1_B, \pi_0 ]} & B }  \end{gathered}\end{equation}
We first observe that $f \vee g$ can also be expressed using the decision of $\llangle f, g \rrangle$. So we compute the following: 
\begin{align*} f \vee g &=~ \llangle f, g \rrangle [1_B, 1_B, \pi_0] \tag*{Def. of $f \vee g$}\\
&=~ \llangle f, g \rrangle \left[ \iota_0 p_0, \iota_1 p_1,  \iota_2 \iota^\circ_2 \pi_0  \right] \tag*{Def. of $p_i$ and Lemma \ref{lem:restcoprodzero}.(\ref{lem:restcoprodzero.iii})} \\
&=~ \llangle f, g \rrangle \left(\iota_0 \oplus  \iota_1 \oplus  \iota_2 \right) \left[  p_0, p_1,  \iota^\circ_2 \pi_0  \right] \tag*{(\ref{eq:copair-comp})}\\
&=~ \mathsf{d}\!\left[ \llangle f, g \rrangle \right] \left(\llangle f, g \rrangle \oplus  \llangle f, g \rrangle \oplus  \llangle f, g \rrangle \right) \left[  p_0, p_1,  \iota^\circ_2 \pi_0  \right] \tag*{\textbf{[D.2]}} \\ 
&=~ \mathsf{d}\!\left[ \llangle f, g \rrangle \right] \left[ \llangle f, g \rrangle p_0, \llangle f, g \rrangle p_1, \llangle f, g \rrangle \iota^\circ_2 \pi_0  \right] \tag*{(\ref{eq:copair-comp})}\\
&=~ \mathsf{d}\!\left[ \llangle f, g \rrangle \right] [ f, g, \langle f,g \rangle \pi_0 ] \tag*{(\ref{diag:pair2}) and Lemma \ref{lem:cprod}.(\ref{lem:cprod.iii})}\\
&=~ \mathsf{d}\!\left[ \llangle f, g \rrangle \right] [ f, g, \overline{g}f ] \tag*{(\ref{diag:pair1})}
\end{align*}
So we have that: 
 \begin{equation}\begin{gathered}\label{def:fgjoin2}
f \vee g := \xymatrixcolsep{5pc}\xymatrix{ A \ar[r]^-{ \mathsf{d}\left[ \llangle f, g \rrangle \right] }  & A \oplus  A \oplus  A \ar[r]^-{[ f, g, \overline{g}f ]} & B }  \end{gathered}\end{equation}
Of course, since $f \smile g$, we also have that $f \vee g = \mathsf{d}\!\left[ \llangle f, g \rrangle \right] [ f, g, \overline{g}f ]$. From this, we then get that $f \vee g = \llangle f, g \rrangle [1_B, 1_B, \pi_1]$ as well. Next, we show that $f \vee g$ is an upper-bound of $f$ and $g$. So we compute: 
\begin{align*} 
\overline{f} (f \vee g) &=~ \overline{f} \llangle f, g \rrangle [ 1_B, 1_B, \pi_0]  \tag{\ref{def:fgjoin1}}\\
&=~ \llangle  \overline{f}f,  \overline{f}g \rrangle [ 1_B, 1_B, \pi_0] \tag*{(\ref{eq:pair-comp})} \\
&=~ \llangle f,  g \rrangle \overline{p_0} [ 1_B, 1_B, \pi_0] \tag*{\textbf{[R.1]} and Lemma \ref{lem:rest-product}.(\ref{lem:rest-product.ii})}  \\
&=~  \llangle f,  g \rrangle (1_A \oplus  0 \oplus  1_{A \times B}) [ 1_B, 1_B, \pi_0] \tag*{Lemma \ref{lem:cprod}.(\ref{lem:cprod.i})} \\
&=~  \llangle f,  g \rrangle [ 1_B, 0, \pi_0] \tag*{(\ref{eq:copair-comp})} \\ 
&=~ \llangle f,g \rrangle p_0 \tag*{Def. of $p_0$}\\
&=~ f \tag*{(\ref{diag:pair2})}
\end{align*}
So $f \leq f \vee g$ and similarly, we can show that $g \leq f \vee g$. Now suppose that we have a map $h: A \to B$ such that $f \leq h$ and $g \leq h$. Then we compute: 
\begin{align*} \overline{f \vee g} h &=~ \overline{ \llangle f, g \rrangle [1_B, 1_B, \pi_0] } h \tag*{Def. of $f \vee g$} \\
&=~ \overline{ \overline{\llangle f, g \rrangle} \llangle f, g \rrangle [1_B, 1_B, \pi_0] } h \tag*{\textbf{[R.1]}}\\
&=~ \overline{\llangle f, g \rrangle}~ \overline{  \llangle f, g \rrangle [1_B, 1_B, \pi_0] } h \tag*{\textbf{[R.3]}}\\
&=~  \overline{  \llangle f, g \rrangle [1_B, 1_B, \pi_0] }~ \overline{\llangle f, g \rrangle} h \tag*{\textbf{[R.2]}}\\
&=~  \overline{ \mathsf{d}\!\left[ \llangle f, g \rrangle \right] [ f, g, \overline{g}f]  } \mathsf{d}\!\left[ \llangle f, g \rrangle \right] \left[ 1_A, 1_A, 1_A \right] h \tag*{(\ref{def:fgjoin2}) and \textbf{[D.1]}}\\
&=~ \mathsf{d}\!\left[ \llangle f, g \rrangle \right] \overline{[ f, g, \overline{g}f] } \left[ h, h, h \right] \tag*{\textbf{[R.4]} and (\ref{eq:copair-comp})}\\
&=~ \mathsf{d}\!\left[ \llangle f, g \rrangle \right] \left( \overline{f} \oplus  \overline{g} \oplus  \overline{\overline{g}f} \right) \left[ h, h, h \right]  \tag*{(\ref{eq:rest-coprod}) and \textbf{[R.3]}} \\
&=~ \mathsf{d}\!\left[ \llangle f, g \rrangle \right] \left[ \overline{f}h, \overline{g}h, \overline{g}\overline{f}h \right]  \tag*{(\ref{eq:copair-comp})}\\ 
&=~ \mathsf{d}\!\left[ \llangle f, g \rrangle \right] \left[ f, g, \overline{g}f \right] \tag*{$f\leq h$ and $g\leq h$}\\ 
&=~ f \vee g
\end{align*}
So $f \vee g \leq h$. Thus we conclude that $f \vee g$ is indeed the join of $f$ and $g$. Next, for any map $k: A^\prime \to A$, we compute:  
\begin{align*} k(f \vee g) &=~ k \llangle f, g \rrangle [ 1_B, 1_B, \pi_0] \tag*{Def.  of $f\vee g$} \\
&=~ \llangle kf, kg \rrangle [ 1_B, 1_B, 0] \tag*{(\ref{eq:pair-comp})} \\
&=~ kf \vee kg \tag*{Def.  of $kf\vee kg$}
\end{align*}
So we conclude that $\mathbb{X}$ is indeed a join restriction category. Now for relative complements, let $f: A \to B$ and ${g: A \to B}$ be maps such that $f \leq g$. Define $g \backslash f: A \to B$ as follows: 
 \begin{equation}\begin{gathered}\label{def:fgcomp}
g \backslash f := \xymatrixcolsep{5pc}\xymatrix{ A \ar[r]^-{ \llangle f, g \rrangle }  & B \oplus  B \oplus  (B \times B) \ar[r]^-{\iota^\circ_1} & B }  \end{gathered}\end{equation}
We first show that $g \backslash f$ and $f$ are disjoint: 
\begin{align*}
\overline{f}(g \backslash f) &=~ \overline{f}  \llangle f, g \rrangle \iota^\circ_1 \tag*{Def. of $g \backslash f$} \\
&=~ \llangle \overline{f}f, \overline{f}g \rrangle \iota^\circ_1 \tag*{(\ref{eq:pair-comp})} \\
&=~  \llangle f,f \rrangle \iota^\circ_1 \tag*{\textbf{[R.1]} and $f \leq g$} \\
&=~ \langle f,f \rangle \iota_2 \iota^\circ_1 \tag*{Lemma \ref{lem:cprod}.(\ref{lem:cprod.v})} \\
&=~ 0 \tag*{Lemma \ref{lem:restcoprodzero}.(\ref{lem:restcoprodzero.iii})}
\end{align*}
So $g \backslash f \perp f$ as desired. To show that the join of $g \backslash f$ and $f$ is equal to $g$, we first have to compute a useful identity. So consider the composite $\llangle \iota^\circ_1, \iota^\circ_2 \pi_1 \rrangle [1_B, 1_B, \pi_0]$. We will show that this composite is equal to $p_1$ by precomposing it with the injection maps:
\begin{align*}
    \iota_0 \llangle \iota^\circ_1, \iota^\circ_2 \pi_1 \rrangle [1_B, 1_B, \pi_0] &=~ \llangle \iota_0\iota^\circ_1, \iota_0\iota^\circ_2 \pi_1 \rrangle [1_B, 1_B, \pi_0] \tag*{(\ref{eq:pair-comp})} \\
    &=~ \llangle 0, 0 \rrangle [1_B, 1_B, \pi_0] \tag*{Lemma \ref{lem:restcoprodzero}.(\ref{lem:restcoprodzero.iii})} \\
    &=~ 0 \tag*{Lemma \ref{lem:cprod}.(\ref{lem:cprod.vi})} 
\end{align*}
\begin{align*} \iota_1 \llangle \iota^\circ_1, \iota^\circ_2 \iota^\circ_2 \pi_1 \rrangle [1_B, 1_B, \pi_0] &=~ \llangle \iota_1\iota^\circ_1, \iota_1\iota^\circ_2 \pi_1 \rrangle [1_B, 1_B, \pi_0] \tag*{(\ref{eq:pair-comp})} \\
&=~  \llangle 1_B, 0 \rrangle [1_B, 1_B, \pi_0] \tag*{Lemma \ref{lem:restcoprodzero}.(\ref{lem:restcoprodzero.iii})} \\
    &=~ \iota_1 [1_B, 1_B, \pi_0] \tag*{Lemma \ref{lem:cprod}.(\ref{lem:cprod.iv})}  \\
    &=~ 1_B \tag*{Def. of $[-,-,-]$}
\end{align*}
\begin{align*} \iota_2 \llangle \iota^\circ_1, \iota^\circ_2 \pi_1 \rrangle [1_B, 1_B, \pi_0] &=~ \llangle \iota_2\iota^\circ_1, \iota_2\iota^\circ_2 \pi_1 \rrangle [1_B, 1_B, \pi_0] \tag*{(\ref{eq:pair-comp})} \\
&=~  \llangle 0, \pi_1 \rrangle [1_B, 1_B, \pi_0] \tag*{Lemma \ref{lem:restcoprodzero}.(\ref{lem:restcoprodzero.iii})} \\
    &=~ \pi_1 \iota_1 [1_B, 1_B, \pi_0] \tag*{Lemma \ref{lem:cprod}.(\ref{lem:cprod.iv})}  \\
    &=~ \pi_1 \tag*{Def. of $[-,-,-]$}
\end{align*}
So by the couniversal property of the coproduct, $\llangle \iota^\circ_1, p_1 \rrangle [1_B, 1_B, \pi_0] = [0,1_B, \pi_1]$. In other words:
\begin{align}\label{eq:p1useful}
    \llangle \iota^\circ_1, \iota^\circ_2 \pi_1 \rrangle [1_B, 1_B, \pi_0] = p_1
\end{align}
Using this identity, we then compute: 
\begin{align*}
g \backslash f \vee f &=~ \llangle g \backslash f, f \rrangle [1_B, 1_B, \pi_0] \tag*{Def. of $g \backslash f \vee f$} \\
&=~ \left \llangle \llangle f,g \rrangle\iota^\circ_1, \overline{f} g \right \rrangle [1_B, 1_B, \pi_0] \tag*{Def. of $g \backslash f$ and $f \leq g$} \\
&=~ \left \llangle \llangle f,g \rrangle\iota^\circ_1,  \llangle f, \overline{f}g \rrangle p_1  \right \rrangle [1_B, 1_B, \pi_0] \tag*{(\ref{diag:pair2})} \\
&=~  \left \llangle \llangle f,g \rrangle\iota^\circ_1,  \llangle f, g \rrangle \overline{p_0} p_1  \right \rrangle [1_B, 1_B, \pi_0] \tag*{Lemma \ref{lem:rest-product}.(\ref{lem:rest-product.ii})} \\ 
&=~   \llangle f,g \rrangle \left \llangle \iota^\circ_1,   \overline{p_0} p_1  \right \rrangle [1_B, 1_B, \pi_0] \tag*{(\ref{eq:pair-comp})} \\
&=~  \llangle f,g \rrangle \left \llangle \iota^\circ_1,   \iota^\circ_2 \pi_1   \right \rrangle [1_B, 1_B, \pi_0]  \tag*{Lemma \ref{lem:cprod}.(\ref{lem:cprod.ii})}  \\
&=~ \llangle f,g \rrangle p_1 \tag*{(\ref{eq:p1useful})} \\
&=~ g \tag*{(\ref{diag:pair2})}
\end{align*}
So we have that $g \backslash f$ is the relative complement of $f$ with respect to $g$. Thus we conclude that $\mathbb{X}$ is a classical restriction category. 
\end{proof}

So we conclude this section by stating the main result of this paper: 

\begin{theorem}\label{thm:cdrc} A distributive restriction category is classical if and only if it has classical products. 
\end{theorem}

Recall that while a distributive restriction category always has $A \& B$ as a tensor product, it may not be a classical product.  An extreme example, already alluded to in the introduction and discussed further below, is an ordinary distributive category viewed as a trivial restriction category (i.e. with $\overline{f} = 1_A$ for all $f$).  This is not a classical restriction category unless it is essentially the final category.

A more subtle example of a distributive restriction category in which the tensor $A \& B$ is not a classical product is the category of topological spaces and partial continuous maps defined on open sets, $\mathsf{TOP}_\bullet$. This is a distributive restriction category with joins, however, it is not classical. Indeed, in $\mathsf{TOP}_\bullet$, the restriction idempotents of a topological space $X$ correspond to its open subsets $U \subseteq X$. As such, the composition of restriction idempotents corresponds to the intersection of opens, while the join will correspond to the union. So if $\mathsf{TOP}_\bullet$ was classical, then for each open subset $U \subseteq X$, there would exist another open subset $U^c \subseteq X$ such that $U \perp U^c$, which means that $U \cap U^c = \emptyset$, and $U \cup U^c = X$. Clearly, this means that $U^c$ must be the set-theoretic complement of $U$. However, the complement of an open subset is not necessarily open. Therefore, not every restriction idempotent in $\mathsf{TOP}_\bullet$ has a complement, and so $\mathsf{TOP}_\bullet$ is not classical. Thus, $\&$ is not a product in $\mathsf{TOP}_\bullet$. 

It is worth mentioning that if we instead consider $\mathsf{TOP}^{clopen}_\bullet$, the subcategory of partial continuous functions defined on \emph{clopen} sets, then $\mathsf{TOP}^{clopen}_\bullet$ is a classical distributive restriction category, and so $\&$ is a product in $\mathsf{TOP}^{clopen}_\bullet$. This leads to another interesting example $\mathsf{STONE}^{clopen}_\bullet$, the category of Stone spaces (which recall are totally disconnected compact Hausdorff topological spaces) and partial continuous functions defined on clopen sets. Then $\mathsf{STONE}^{clopen}_\bullet$ is also a classical distributive restriction category and, thanks to Stone duality, we also have that the opposite category of Boolean algebras and maps which preserves meets, joins, and the bottom element (but not necessarily the top element) is also a classical distributive restriction category. 

\section{Classical Classification}\label{sec:classified}

In this section, we show that classical distributive restriction categories are in fact precisely the Kleisli categories of the exception monads of distributive categories. For a more in-depth introduction to distributive categories, we invite the reader to see \cite{carboni1993distributive,cockett1993distributive}.  

By \textbf{distributive category} \cite[Sec 3]{cockett1993distributive}, we mean a category $\mathbb{D}$ with finite products and finite coproducts which distribute in the same sense as in Def \ref{def:drc}. From the point of view of restriction categories, a distributive category is a distributive restriction category $\mathbb{D}$ with a trivial restriction, so $\overline{f} =1_A$ and every map is total. As such, for a distributive category $\mathbb{D}$, we will use $\times$ for the product and $\mathsf{1}$ for the terminal object. As noted above, a distributive category (seen as a trivial restriction category) does not have classical products unless the category is trivial.                                                                                                                                                                                                                   The subcategory of total maps of a distributive restriction category, on the other hand, is always a distributive category: 

\begin{lemma} \cite[Prop 5.7]{cockett2007restriction} For a distributive restriction category $\mathbb{X}$, $\mathcal{T}[\mathbb{X}]$ is a distributive category. 
\end{lemma}

For a distributive category $\mathbb{D}$, its \textbf{exception monad} (also sometimes called the maybe monad) is the monad $\_ \oplus \mathsf{1}$, where the unit is the injection $\iota_0: A \to A \oplus \mathsf{1}$ and the multiplication is $[\iota_0, \iota_1, \iota_1]: A  \oplus \mathsf{1}  \oplus \mathsf{1} \to A  \oplus \mathsf{1}$. It is well established that the Kleisli category of the exception monad $\mathbb{D}_{\_ \oplus \mathsf{1}}$ is a restriction category, in fact, it is a distributive restriction category with restriction zeroes. To help us distinguish between maps in the base category and maps in the Kleisli category, we will use interpretation brackets, which are functions on homsets $\llbracket - \rrbracket: \mathbb{D}_{\_ \oplus \mathsf{1}}(A, B) \to \mathbb{D}(A,B \oplus \mathsf{1})$. So a Kleisli map from $A$ to $B$ will be written as $\llbracket f \rrbracket: A \to B \oplus \mathsf{1}$. To demonstrate how these interpretation brackets are useful, here is how to express identity maps and composition in the Kleisli category: 
\begin{align}
    \llbracket 1_A \rrbracket= \iota_0 && \llbracket fg \rrbracket = \llbracket f \rrbracket \left[ \llbracket g \rrbracket , 1_{\mathsf{1}} \right] 
\end{align}
Let us now describe how the Kleisli category of the exception monad is a distributive restriction category with restriction zeroes. 

 \begin{proposition}\label{d+1rest} \cite[Ex 5.4]{cockett2007restriction} Let $\mathbb{D}$ be a distributive category. Then $\mathbb{D}_{\_ \oplus \mathsf{1}}$ is a distributive restriction category with restriction zeroes where: 
     \begin{enumerate}[{\em (i)}]
    \item For a Kleisli map $\llbracket f \rrbracket: A \to B \oplus \mathsf{1}$, its restriction $\left \llbracket~ \overline{f}~ \right \rrbracket: A \to A \oplus \mathsf{1}$ is defined as the following composite: 
     \begin{equation}\begin{gathered}\label{def:rest+1}
\begin{array}[c]{c} \left \llbracket ~\overline{f}~ \right \rrbracket \end{array}
 :=  \begin{array}[c]{c} \xymatrixcolsep{4pc}\xymatrixrowsep{1pc}\xymatrix{ A \ar[r]^-{ \left\langle 1_A, \llbracket f \rrbracket  \right\rangle }  & A \times (B \oplus \mathsf{1}) \cong (A \times B) \oplus (A \times \mathsf{1})  \ar[r]^-{\pi_0 \oplus \pi_1} & A \oplus \mathsf{1} } \end{array}  \end{gathered}\end{equation}
 \item A Kleisli map $\llbracket f \rrbracket: A \to B \oplus \mathsf{1}$ is total if and only if $\llbracket f \rrbracket = g \iota_0$ for some (necessarily unique) map $g: A \to B$ in $\mathbb{D}$. Therefore, we have an isomorphism $\mathcal{T}[\mathbb{D}_{\_ \oplus \mathsf{1}}] \cong \mathbb{D}$.  
 \item The restriction terminal object is the terminal object $\mathsf{1}$, and $\llbracket t_A \rrbracket: A \to \mathsf{1} \oplus \mathsf{1}$ is defined as the composite $\llbracket t_A \rrbracket = t_A \iota_0$. 
 \item The restriction product is the product $A \times B$ where the projections $\llbracket \pi_0 \rrbracket: A \times B \to A \oplus \mathsf{1}$ and $\llbracket \pi_1 \rrbracket: A \times B \to B \oplus \mathsf{1}$ are defined as the composites $\llbracket \pi_j \rrbracket = \pi_j \iota_0$, and where the pairing of Kleisli maps $\llbracket f \rrbracket: C \to A \oplus \mathsf{1}$ and $\llbracket g \rrbracket: C \to B \oplus \mathsf{1}$ is the Kleisli map $\llbracket \langle f,g \rangle \rrbracket: C \to (A \times B) \oplus \mathsf{1}$ defined as follows: 
  \begin{equation}\begin{gathered}\label{def:pair+1}
\begin{array}[c]{c} \llbracket \langle f,g \rangle \rrbracket \end{array}
 :=  \begin{array}[c]{c} \xymatrixcolsep{5pc}\xymatrixrowsep{1pc}\xymatrix{ C \ar[r]^-{ \left\langle \llbracket f \rrbracket, \llbracket g \rrbracket \right\rangle }  & ( A \oplus \mathsf{1}) \times (B \oplus \mathsf{1}) } \\
\xymatrixcolsep{5pc}\xymatrixrowsep{1pc}\xymatrix{  \cong A \oplus B \oplus (A \times B) \oplus \mathsf{1} \ar[r]^-{\left[t_A \iota_1, t_B \iota_1, \iota_0, \iota_1 \right]} & (A \times B) \oplus \mathsf{1} } \end{array}  \end{gathered}\end{equation}
 \item The restriction initial object is the initial object $\mathsf{0}$ where $\llbracket z_A \rrbracket: \mathsf{0} \to A \oplus \mathsf{1}$ is defined as $\llbracket z_A \rrbracket := z_{A \oplus \mathsf{1}}$.
 \item The restriction coproduct is the coproduct $A_0 ~ \oplus ~ \hdots ~ \oplus ~ A_n$ where the injections ${\llbracket \iota_j \rrbracket: A_j \to (A_0 \oplus \hdots \oplus A_n) \oplus \mathsf{1}}$ are defined as $\llbracket \iota_j \rrbracket := \iota_j \iota_1$, and where the copairing of Kleisli maps is given by the copairing in the base category, $\left \llbracket [f_0, \hdots, f_n] \right \rrbracket = \left [ \llbracket f_0 \rrbracket, \hdots, \llbracket f_n \rrbracket \right] $. 
 \item The restriction zero maps $\llbracket 0 \rrbracket: A \to B \oplus \mathsf{1}$ are defined as the composite $\llbracket 0 \rrbracket := t_A \iota_1$. 
\end{enumerate}
Therefore, $\mathbb{D}_{\_ \oplus \mathsf{1}}$ is also an extensive restriction category. 
 \end{proposition}

We will now show that $\mathbb{D}_{\_ \oplus \mathsf{1}}$ is also classical, which is a novel observation. To do so, we will explain why $\mathbb{D}_{\_ \oplus \mathsf{1}}$ has classical products. For starters, it is already known that the Kleisli category of the exception monad has products. Indeed, since by distributivity we have that $(A \times B) \oplus \mathsf{1} \cong A \oplus B \oplus (A \times B) \oplus \mathsf{1}$, it follows that products in $\mathbb{D}_{\_ \oplus \mathsf{1}}$ are of the form $A \oplus B \oplus (A \times B)$. More explicitly: 

 \begin{lemma}\label{lem:prod+1} \cite[Prop 3.4]{cockett1997weakly} Let $\mathbb{D}$ be a distributive category. Then $\mathbb{D}_{\_ \oplus \mathsf{1}}$ has finite products where: 
          \begin{enumerate}[{\em (i)}]
          \item The terminal object is $\mathsf{0}$ and where $\llbracket !_A \rrbracket: A \to \mathsf{0} \oplus \mathsf{1}$ is defined as $\llbracket !_A \rrbracket := t_A \iota_1$.
 \item The binary product $\&$ is defined as $A \& B := A \oplus B \oplus (A \times B)$ and where the projections $\llbracket p_0 \rrbracket: A \oplus B \oplus (A \times B) \to A \oplus \mathsf{1}$ and $\llbracket p_1 \rrbracket: A \oplus B \oplus (A \times B) \to B \oplus \mathsf{1}$ defined as:
 \begin{align}
     \llbracket p_0 \rrbracket = [\iota_0, t_B \iota_1, \pi_0 \iota_0] &&  \llbracket p_0 \rrbracket = [t_A \iota_1, \iota_0, \pi_1 \iota_0] 
 \end{align}
The pairing of Kleisli maps $\llbracket f \rrbracket: C \to A \oplus \mathsf{1}$ and $\llbracket g \rrbracket: C \to B \oplus \mathsf{1}$ is the Kleisli map $\llbracket \llangle f,g \rrangle \rrbracket: C \to \left(A \oplus B \oplus (A \times B) \right) \oplus \mathsf{1}$ defined as follows: 
  \begin{equation}\begin{gathered}\label{def:cpair+1}
\begin{array}[c]{c} \llbracket \llangle f,g \rrangle \rrbracket \end{array}
 :=  \begin{array}[c]{c} \xymatrixcolsep{5pc}\xymatrixrowsep{1pc}\xymatrix{ C \ar[r]^-{ \left\langle \llbracket f \rrbracket, \llbracket g \rrbracket \right\rangle }  & ( A \oplus \mathsf{1}) \times (B \oplus \mathsf{1}) \cong A \oplus B \oplus (A \times B) \oplus \mathsf{1} } \end{array}  \end{gathered}\end{equation}
\end{enumerate}
 \end{lemma}

To show that we in fact have classical products, it remains to show that the projections are in fact the classical projections. 

 \begin{proposition}\label{prop:d+1class} Let $\mathbb{D}$ be a distributive category. Then $\mathbb{D}_{\_ \oplus \mathsf{1}}$ has classical products, and therefore, $\mathbb{D}_{\_ \oplus \mathsf{1}}$ is a classical distributive restriction category.
 \end{proposition}
 \begin{proof} We need only check that the projections $\llbracket p_0 \rrbracket$ and $\llbracket p_1 \rrbracket$ defined in Lemma \ref{lem:prod+1} are defined as the classical projections in Def \ref{def:classprod}. So we quickly check: 
 \begin{align*}
     \llbracket p_0 \rrbracket &=~  [\iota_0, t_B \iota_1, \pi_0 \iota_0] \tag*{Def. of $p_0$ in $\mathbb{D}_{\_ \oplus \mathsf{1}}$} \\
     &=~\left[ \llbracket 1_A \rrbracket, \llbracket 0 \rrbracket, \llbracket \pi_0 \rrbracket \right] \tag*{Def. of $1_A$, $0$, and $\pi_0$ in $\mathbb{D}_{\_ \oplus \mathsf{1}}$} \\
     &=~ \left \llbracket [1_A, 0, \pi_0] \right \rrbracket \tag*{Def. of $[-,-,-]$ in $\mathbb{D}_{\_ \oplus \mathsf{1}}$} 
 \end{align*}
 So $\llbracket p_0 \rrbracket = \left \llbracket [1_A, 0, \pi_0] \right \rrbracket$, and similarly we can easily show that $\llbracket p_1 \rrbracket = \left \llbracket [0, 1_B \pi_1] \right \rrbracket$. Therefore, $p_0$ and $p_1$ are indeed the classical projections in $\mathbb{D}_{\_ \oplus \mathsf{1}}$, so we conclude that $\mathbb{D}_{\_ \oplus \mathsf{1}}$ has classical products. Therefore, by Theorem \ref{thm:cdrc},  $\mathbb{D}_{\_ \oplus \mathsf{1}}$ is a classical distributive restriction category. 
 \end{proof}

 We now wish to show the converse that a classical distributive restriction category is the Kleisli category for an exception monad of a distributive category. To do so, we must first discuss classified restriction categories, which are restriction categories that arise as Kleisli categories. The key characteristic of a classified restriction category is that every map can be factored as a total map followed by a natural partial counit. This natural partial counit needs to be a \emph{restriction retraction} \cite[Sec 3.1]{cockett2003restriction}, that is, a map $f: A \to B$ such that there is another map $g: B \to A$ such that $gf =1_B$ and $fg = \overline{f}$. 

 \begin{definition}\label{def:restclass} A \textbf{classified restriction category} \cite[Sec 3.2]{cockett2003restriction} is a restriction category $\mathbb{X}$ such that for each object $A$, there is an object $\mathcal{R}(A)$, called the \textbf{classifier}, and a restriction retraction $\varepsilon_A: \mathcal{R}(A) \to A$, called the \textbf{classifying map}, such that for every map $f: A \to B$, there exists a unique total map $\mathcal{T}(f): A \to \mathcal{R}(B)$ such that the following diagram commutes: 
     \begin{equation}\begin{gathered}\label{diag:Rclassified} 
  \xymatrixcolsep{5pc}\xymatrix{ A \ar[dr]_-{\mathcal{T}(f)} \ar[rr]^-{f}  & & B \\
  & \mathcal{R}(B) \ar[ur]_-{\varepsilon_B} }  \end{gathered}\end{equation}
\end{definition}

The classifier induces a monad on the subcategory of total maps. The Kleisli category of this monad is not only a restriction category, but also isomorphic as a restriction category to the starting classified restriction category. By an isomorphism of restriction categories, we mean an isomorphism of categories which preserves the restriction. In the following proposition, we again use the special interpretation brackets for Kleisli categories. 

\begin{proposition}\label{prop:classkleisli} \cite[Prop 3.10]{cockett2003restriction} Let $\mathbb{X}$ be a classified restriction category, with classifier $\mathcal{R}$ and classifying map $\varepsilon$. Then there is a monad $(\mathcal{R}, \mu, \eta)$ on $\mathcal{T}[\mathbb{X}]$ defined as follows:
\begin{align}
    \mathcal{R}(A) = \mathcal{R}(A) && \mathcal{R}(f) = \mathcal{T}(\varepsilon_A f) && \mu_A = \mathcal{R}(\varepsilon_A) && \eta_A = \mathcal{T}(1_A)
\end{align}
    Furthermore, the Kleisli category $\mathcal{T}[\mathbb{X}]_\mathcal{R}$ is a restriction category where for a Kleisli map $\llbracket f \rrbracket: A \to \mathcal{R}(B)$, its restriction $\llbracket ~ \overline{f} ~ \rrbracket: A \to \mathcal{R}(A)$ is defined as $\llbracket ~ \overline{f} ~ \rrbracket := \mathcal{T}\left( \overline{\llbracket f \rrbracket \varepsilon_B} \right)$, where the restriction operator on the right-hand side is that of $\mathbb{X}$. Moreover, the functor $\mathcal{T}_\sharp: \mathbb{X} \to \mathcal{T}[\mathbb{X}]_\mathcal{R}$, which is defined as follows: 
    \begin{align}
        \mathcal{T}_\sharp(A) = A && \llbracket \mathcal{T}_\sharp(f) \rrbracket = \mathcal{T}(f)
    \end{align}
is an isomorphism of restriction categories with inverse $\mathcal{T}^{-1}_\sharp: \mathcal{T}[\mathbb{X}]_\mathcal{R} \to \mathbb{X}$ defined as follows: 
\begin{align}
    \mathcal{T}^{-1}_\sharp(A) = A && \mathcal{T}^{-1}_\sharp(f) = \llbracket f \rrbracket \varepsilon_B 
\end{align}
So $\mathbb{X} \cong \mathcal{T}[\mathbb{X}]_\mathcal{R}$. 
\end{proposition}

Requiring that a distributive restriction category is the Kleisli category of an exception monad, it needs to be classified with the classifier being of the form $\mathcal{R}(\_) = \_ \oplus \mathsf{1}$. As such, we introduce the notion of being classically classified. 

 \begin{definition}\label{def:classclass} A distributive restriction category $\mathbb{X}$ is said to be \textbf{classically classified} if $\mathbb{X}$ has restriction zeroes and for every map $f: A \to B$, there exists a unique total map $\mathcal{T}(f): A \to B \oplus  \mathsf{1}$ such that the following diagram commutes: 
 \begin{equation}\begin{gathered}\label{diag:classified} 
  \xymatrixcolsep{5pc}\xymatrix{ A \ar[dr]_-{\mathcal{T}(f)} \ar[rr]^-{f}  & & B \\
  & B \oplus  \mathsf{1} \ar[ur]_-{\iota^\circ_0} }  \end{gathered}\end{equation}
\end{definition}

The nomenclature is justified since we will show that being classically classified is equivalent to being classical (Prop \ref{prop:classified2}). It is straightforward to see that being classically classified is a special case of being classified and that the resulting monad is the exception monad on the subcategory of total maps. 

\begin{proposition}\label{prop:classified1} Let $\mathbb{X}$ be a distributive restriction category which is classically classified. Then $\mathbb{X}$ is classified, where $\mathcal{R}(A) = A \oplus \mathsf{1}$ and $\varepsilon_A = \iota^\circ_0$. Furthermore, the induced monad on $\mathcal{T}[\mathbb{X}]$ is precisely the exception monad $\_ \oplus \mathsf{1}$ and the induced restriction structure on $\mathcal{T}[\mathbb{X}]_{\_ \oplus \mathsf{1}}$ is precisely that defined in Prop \ref{d+1rest}. Therefore, $\mathbb{X}$ is restriction isomorphic to $\mathcal{T}[\mathbb{X}]_{\_ \oplus \mathsf{1}}$. 
\end{proposition}

Here are our main classically classified examples:  

\begin{example} \normalfont  $\mathsf{PAR}$ is classically classified where for a partial function $f: X \to Y$, $\mathcal{T}(f): X \to Y \sqcup \lbrace \ast \rbrace$ is the total function defined as follows: 
\[ \mathcal{T}(f)(x) = \begin{cases} x & \text{if } f(x) \downarrow  \\
\ast & \text{if } f(x) \uparrow  \end{cases} \]
Therefore, since $\mathcal{T}\left[ \mathsf{PAR} \right] = \mathsf{SET}$, we recover the well-known result that $\mathsf{SET}_{\_ \sqcup \lbrace \ast \rbrace} \cong \mathsf{PAR}$. 
\end{example}

\begin{example} \normalfont  $k\text{-}\mathsf{CALG}^{op}_\bullet$ is classically classified, so $k\text{-}\mathsf{CALG}_\bullet$ is coclassically coclassified where for a non-untial $k$-algebra morphism $f: A \to B$, $\mathcal{T}(f): A \times k \to B$ is the $k$-algebra morphism defined as follows: 
\[  \mathcal{T}(f)(a,r) = f(a) + r - rf(1) \]
Therefore, since $\mathcal{T}[k\text{-}\mathsf{CALG}^{op}_\bullet] = k\text{-}\mathsf{CALG}^{op}$, we have that $k\text{-}\mathsf{CALG}^{op}_{\_ \times k} \cong k\text{-}\mathsf{CALG}^{op}_\bullet$. In other words, the coKleisli category of the comonad $\_ \times k$ on $k\text{-}\mathsf{CALG}$ is $k\text{-}\mathsf{CALG}_\bullet$, so $k\text{-}\mathsf{CALG}_{\_ \times k} \cong k\text{-}\mathsf{CALG}_\bullet$. 
\end{example}

Also observe that $\mathsf{TOP}^{clopen}_\bullet$ and $\mathsf{STONE}^{clopen}_\bullet$ are classically classified and are respectively the Kleisli categories of the exception monads on $\mathsf{TOP}$, the category of topological spaces and continuous functions, and $\mathsf{STONE}$, the category of Stone spaces and continuous functions. More generally of course, the Kleisli categories of exception monads are also classically classified: 

  \begin{proposition} \cite[Ex 3.18]{cockett2003restriction} Let $\mathbb{D}$ be a distributive category. Then $\mathbb{D}_{\_ \oplus \mathsf{1}}$ is classically classified, where for a Kleisli map $\llbracket f \rrbracket: A \to B \oplus \mathsf{1}$, $\llbracket \mathcal{T}(f) \rrbracket: A \to (B \oplus \mathsf{1}) \oplus \mathsf{1}$ is defined as $\llbracket \mathcal{T}(f) \rrbracket := \llbracket f \rrbracket \iota_0$. 
 \end{proposition}

As such, we obtain that a distributive restriction category is classically classified if and only if it is restriction equivalent to the Kleisli category of an exception monad. By a restriction equivalence, we mean an equivalence of category where the functors preserve the restriction and both the unit and counit of the adjunction are total. 

\begin{corollary}\label{cor:classclass} A distributive restriction category $\mathbb{X}$ is classically classified if and only if there is a distributive category $\mathbb{D}$ such that $\mathbb{X}$ is restriction equivalent to $\mathbb{D}_{\_ \oplus \mathsf{1}}$. 
\end{corollary}

It follows that being classically classified implies that one has classical products: 

 \begin{proposition}\label{prop:classified2} A distributive restriction category which is classically classified is a classical distributive restriction category, and therefore has classical products. 
\end{proposition}
\begin{proof} Let $\mathbb{X}$ be a distributive restriction category and suppose that $\mathbb{X}$ is classically classified. Per Prop \ref{prop:classified1}, $\mathbb{X}$ is restriction isomorphic to $\mathcal{T}[\mathbb{X}]_{\_ \oplus \mathsf{1}}$. However by Prop \ref{prop:d+1class}, $\mathcal{T}[\mathbb{X}]_{\_ \oplus \mathsf{1}}$ is classical and has classical products. Clearly, a restriction isomorphism transfers being classical and having classical products from one to the other. Therefore, we conclude that $\mathbb{X}$ is classical and also has classical products. 
\end{proof}

Finally, we are also able to prove the converse: 

\begin{proposition} A classical distributive restriction category is classically classified.     
\end{proposition}
\begin{proof} Let $\mathbb{X}$ be a classical distributive restriction category. Since $\mathbb{X}$ has restriction zeros by assumption, it remains to show that every map factors as in (\ref{diag:classified}). So consider an arbitrary map $f: A \to B$, and also the composites $f\iota_0: A \to B \oplus \mathsf{1}$ and $\overline{f}^c t_A \iota_1: A \to B \oplus \mathsf{1}$. We first show they are disjoint: 
\begin{align*}
    \overline{f\iota_0} \overline{f}^c t_A \iota_1 &=~  \overline{f} \overline{f}^c t_A \iota_1 \tag*{Lemma \ref{lem:rest}.(\ref{lem:rest.total.2}) and $\iota_0$ total} \\ 
    &=~ 0 \tag*{Lemma \ref{lem:ecomp}.(\ref{lem:ecomp.ii})} 
\end{align*}
So $f\iota_0 \perp \overline{f}^c t_A \iota_1$. Thus define the map $\mathcal{T}(f): A \to B \oplus  \mathsf{1}$ as the join of these disjoint composites: 
\begin{align}
    \mathcal{T}(f) := f\iota_0 \vee \overline{f}^c t_A \iota_1
\end{align}
 First we check that $\mathcal{T}(f)$ is total. So we compute: 
 \begin{align*}
  \overline{\mathcal{T}(f)} &=~  \overline{f\iota_0 \vee \overline{f}^c t_A \iota_1} \tag*{Def. of $\mathcal{T}$} \\
  &=~ \overline{f \iota_0} \vee \overline{\overline{f}^c t_A \iota_1} \tag*{Lemma \ref{lem:join}.(\ref{lem:join.iv})} \\
  &=~ \overline{f} \vee \overline{\overline{f}^c} \tag*{Lemma \ref{lem:rest}.(\ref{lem:rest.total.2}), and $\iota_j$ and $t_A$ total} \\ 
  &=~ \overline{f} \vee \overline{f}^c \tag*{Lemma \ref{lem:ecomp}.(\ref{lem:ecomp.0})} \\
  &=~ 1_A \tag*{Lemma \ref{lem:ecomp}.(\ref{lem:ecomp.ii})} 
 \end{align*}
Next we check that (\ref{diag:classified}) holds: 
\begin{align*} 
\mathcal{T}(f) \iota^\circ_0 &=~ (f\iota_0 \vee \overline{f}^c t_A \iota_1) \iota^\circ_0 \tag*{Def. of $\mathcal{T}$} \\
&=~ f\iota_0\iota^\circ_0 \vee \overline{f}^c t_A \iota_1\iota^\circ_0  \tag*{Lemma \ref{lem:join}.(\ref{lem:join.iii})} \\ 
&=~ f \vee 0 \tag*{Lemma \ref{lem:restcoprodzero}.(\ref{lem:restcoprodzero.iii})}\\
&=~ f \tag*{Lemma \ref{lem:join}.(\ref{lem:join.ii})} 
\end{align*}
 Lastly we need to show the uniqueness of $\mathcal{T}(f)$. So let $g: A \to B\oplus  \mathsf{1}$ be a total map such that $g \iota_0^\circ = f$. We first compute that: 
 \begin{align*}
   g &=~  g (\iota^\circ_0\iota_0 \vee \iota^\circ_1 \iota_1) \tag*{Lemma \ref{lem:class-coprod}.(\ref{lem:class-coprod.iv})} \\ 
     &=~ g\iota^\circ_0\iota_0 \vee g\iota^\circ_1 \iota_1 \tag*{\textbf{[J.3]}} \\
     &=~ f \iota_0 \vee g\iota^\circ_1 \iota_1 \tag*{Assump. on $g$}
 \end{align*}
 So we have that: 
 \begin{align}\label{galmost}
     g = f \iota_0 \vee g\iota^\circ_1 \iota_1
 \end{align}
Thus it remains to show that $g\iota^\circ_1= \overline{f}^c t_A$. To do so, we will prove that $\overline{f}^c$ is equal to $\overline{g\iota^\circ_1}$. We first compute that: 
 \begin{align}
     \overline{f}~ \overline{g \iota^\circ_1} &=~ \overline{\overline{f} g\iota^\circ_1} \tag*{\textbf{[R.3]}} \\
     &=~ \overline{\overline{f \iota_0}g\iota^\circ_1} \tag*{Lemma \ref{lem:rest}.(\ref{lem:rest.total.2}) and $\iota_0$ total} \\ 
     &=~ \overline{\overline{f \iota_0}\left(f \iota_0 \vee g\iota^\circ_1 \iota_1 \right)\iota^\circ_1} \tag*{(\ref{galmost})} \\
     &=~ \overline{f \iota_0 \iota^\circ_1} \tag*{Lemma \ref{lem:join}.(\ref{lem:join.v})} \\
     &=~ \overline{0} \tag*{Lemma \ref{lem:join}.(\ref{lem:join.iii})} \\
     &=~ 0 \tag*{Rest. zero}
 \end{align}
 So $ \overline{f}~ \overline{g \iota^\circ_1}=0$. Next we compute: 
 \begin{align*}
      \overline{f} \vee \overline{g \iota^\circ_1} &=~ \overline{f \iota_0} \vee \overline{g \iota^\circ_1\iota_1}  \tag*{Lemma \ref{lem:rest}.(\ref{lem:rest.total.2}) and $\iota_j$ total} \\ 
      &=~ \overline{f \iota_0 \vee g\iota^\circ_1 \iota_1} \tag*{Lemma \ref{lem:join}.(\ref{lem:join.iv})} \\
      &=~ \overline{g} \tag*{(\ref{galmost})} \\
      &=~ 1_A \tag*{$g$ total}
 \end{align*}
 So $  \overline{f} \vee \overline{g \iota^\circ_1}=1_A$. Thus by Lemma \ref{lem:ecomp}.(\ref{lem:ecomp.iv}), it follows that:
 \begin{align} \label{gfc}
    \overline{f}^c = \overline{g\iota^\circ_1} 
 \end{align} 
Lastly, recall that in a Cartesian restriction category, for any map of type $h: A \to \mathsf{1}$, we have that $h = \overline{f}t_A$ \cite[Prop 2.8]{cockett2012differential}. Since $g\iota^\circ_1: A \to \mathsf{1}$, we then also have that: 
 \begin{align} \label{git}
    g\iota^\circ_1 = \overline{g\iota^\circ_1}t_A
 \end{align} 
Thus we finally compute that: 
 \begin{align*}
g &=~  f \iota_0 \vee g\iota^\circ_1 \iota_1  \tag*{(\ref{galmost})} \\
&=~ f\iota_0 \vee \overline{g\iota^\circ_1} t_A \iota_1 \tag*{(\ref{git})} \\
&=~ f\iota_0 \vee \overline{f}^c t_A \iota_1 \tag*{(\ref{gfc})}\\
&=~ \mathcal{T}(f) \tag*{Def. of $\mathcal{T}(f)$}
 \end{align*}
So $\mathcal{T}(f)$ is unique. Therefore we conclude that $\mathbb{X}$ is classically classified. 
 \end{proof}

\begin{corollary}If $\mathbb{X}$ is a classical distributive restriction category, then we have a restriction isomorphism $\mathbb{X} \cong \mathcal{T}\left[ \mathbb{X} \right]_{\_ \oplus \mathsf{1}}$. 
\end{corollary} 

Bringing all of these results together, we conclude this paper with the following statement characterizing classical distributive restriction categories. 

 \begin{theorem}\label{thm2} For a distributive restriction category $\mathbb{X}$, the following are equivalent: 
 \begin{enumerate}[{\em (i)}]
\item $\mathbb{X}$ is classical;
\item $\mathbb{X}$ has classical products;
\item $\mathbb{X}$ is classically classified;
\item There is a distributive category $\mathbb{D}$ such that $\mathbb{X}$ is restriction equivalent to $\mathbb{D}_{\_ \oplus \mathsf{1}}$. 
\end{enumerate}
\end{theorem}

\bibliographystyle{plain}      
\bibliography{cdrcbib}   

\end{document}